\documentclass[11pt]{amsart}
\usepackage[left=1.25in,right=1.25in]{geometry}
\usepackage{amsthm}
\usepackage{amsmath}
\usepackage{mathtools}
\usepackage{amssymb}
\usepackage{tikz-cd}
\usepackage{tikz}
\usepackage{array}
\usetikzlibrary{decorations.markings}
\usetikzlibrary{patterns}
\usepackage{enumitem}
\usepackage{amsfonts}
\usepackage{todonotes}
\usepackage[utf8]{inputenc}
\usepackage{ytableau}
\usepackage{subcaption}
\usepackage{caption}
\usepackage[export]{adjustbox}
\usepackage{mathdots}
\usepackage{comment}
\usepackage{hyperref}
\usepackage{cleveref}
\usetikzlibrary{calc, shapes, backgrounds,arrows,positioning,decorations.pathmorphing}

\newtheorem{theorem}{Theorem}[section]
\newtheorem{prop}[theorem]{Proposition}

\newtheorem{lemma}[theorem]{Lemma}
\newtheorem{cor}[theorem]{Corollary}

\theoremstyle{definition}
\newtheorem{ex}[theorem]{Example}

\newtheorem{defin}[theorem]{Definition}
\newtheorem{remark}[theorem]{Remark}

\theoremstyle{plain}

\makeatletter
\newcommand\circled[1]{%
  \mathpalette\@circled{#1}%
}
\newcommand\@circled[2]{%
  \tikz[baseline=(math.base)] \node[draw,circle,inner sep=1pt] (math) {$\m@th#1#2$};%
}
\newcommand{\gr}{\mathrm{Gr}}
\newcommand{\RNum}[1]{\uppercase\expandafter{\romannumeral #1\relax}}
\makeatother
\newcommand{\shiliang}[1]{{\color{blue} \sf Shiliang: [#1]}}
\newcommand{\jiyang}[1]{{\color{red} \sf Jiyang: [#1]}}

\DeclareMathOperator{\wt}{wt}
\DeclareMathOperator{\depth}{depth}

\title{Quantum Bruhat graphs and tilted Richardson varieties}

\author{Jiyang Gao}
\address{Department of Mathematics, Harvard University, Cambridge, MA 02138}
\email{\href{mailto:jgao@math.harvard.edu}{{\tt jgao@math.harvard.edu}}}

\author{Shiliang Gao}
\address{Department of Mathematics, University of Illinois Urbana-Champaign, Urbana, IL 61801}
\email{\href{mailto:sgao23@illinois.edu}{{\tt sgao23@illinois.edu}}}

\author{Yibo Gao}
\address{Beijing International Center for Mathematical Research, Peking University, Beijing 100084}
\email{\href{mailto:gaoyibo@bicmr.pku.edu.cn}{{\tt gaoyibo@bicmr.pku.edu.cn}}}

\date{\today}

\begin{document}
\begin{abstract}
Quantum Bruhat graph is a weighted directed graph on a finite Weyl group first defined by Brenti-Fomin-Postnikov. It encodes quantum Monk's rule and can be utilized to study the $3$-point Gromov-Witten invariants of the flag variety. In this paper, we provide an explicit formula for the minimal weights between any pair of permutations on the quantum Bruhat graph, and consequently obtain an Ehresmann-like characterization for the tilted Bruhat order. Moreover, for any ordered pair of permutations $u$ and $v$, we define the tilted Richardson variety $\mathcal{T}_{u,v}$, with a stratification that gives a geometric meaning to intervals in the tilted Bruhat order. We provide a few equivalent definitions to this new family of varieties that include Richardson varieties, and establish some fundamental geometric properties including their dimensions and closure relations. 
\end{abstract}
\maketitle
\tableofcontents
\section{Introduction}\label{sec:intro}
Hilbert's fifteenth problem, Schubert calculus, concerns the full flag variety \[\mathrm{Fl}_n=\{0\subset V_1\subset \cdots\subset V_{n-1}\subset \mathbb{C}^n\:|\: \dim V_i=i\text{ for }i=1,\ldots,n-1\}\]
and its Schubert decomposition $\mathrm{Fl}_n=\bigsqcup_{w\in S_n} X_w^{\circ}$. The cohomology ring $H^*(\mathrm{Fl}_n)$ has a linear basis given by the Schubert varieties $\{[X_w]\}_{w\in S_n}$. The corresponding structure constants $c_{u,w}^v$'s, also referred to as the \emph{generalized Littlewood-Richardson numbers}, are known to be nonnegative from transversal intersection. It has been a long standing open problem to find a combinatorial interpretation of them. The study of flag varieties, Schubert varieties and these structure constants is central in algebraic geometry and algebraic combinatorics.

The small quantum cohomology ring $QH^*(\mathrm{Fl}_n)$ is a deformation of the cohomology ring. It shows up naturally in theoretical physics. The structure constants of $QH^*(\mathrm{Fl}_n)$ with respect to the Schubert basis are known as the \emph{3-point genus-0 Gromov-Witten invariants}. They extend the generalized Littlewood-Richardson numbers in ``quantum" direction. 

The problem of multiplying Schubert classes in $QH^*(\mathrm{Fl}_n)$ can be naturally encoded via the \emph{quantum Bruhat graph}, first defined by Brenti-Fomin-Postnikov \cite{BFP-tilted-Bruhat} and utilized by Postnikov \cite{Postnikov-quantum-Bruhat-graph}. The quantum Bruhat graph can be seen as a graphical representation of the \emph{quantum Monk's rule} and enjoys very rich algebraic and combinatorial properties. In particular,  the minimal degree $q^d$ that appears in the quantum product $[X_u]*[X^v]$ is the weight of any shortest directed path from $u$ to $v$ \cite{Postnikov-quantum-Bruhat-graph}. The quantum Bruhat graph directly gives rise to the \emph{tilted Bruhat order} \cite{BFP-tilted-Bruhat}, and these are our main combinatorial objects of interest for this paper.

\

\textbf{Main result 1:} weights in the quantum Bruhat graph:
\begin{enumerate}
\item An explicit combinatorial formula for the minimal weight between any pair of permutations $u$ to $v$ (Theorem~\ref{thm:weight-distance}).
\item An Ehresmann-like characterizaiton for the tilted Bruhat order (Theorem~\ref{thm:tilted-criterion}).
\end{enumerate}

\

We remark that Theorem~\ref{thm:weight-distance} was also obtained via a combination of Postnikov's toric Schur polynomials \cite{Postnikov-quantum-Schur} on the quantum cohomology ring of the Grassmannian, and a geometric result by Buch-Chung-Li-Mihalcea \cite{bclm2020-quantumK}. See also \cite{FW}. Our proof is independent and combinatorial. 

While weights on the quantum Bruhat graph encode important information in the quantum cohomology of the flag variety, we present a novel geometric interpretation of intervals in the tilted Bruhat order with a more classical flavor. For any $u,v\in S_n$, we define the \emph{tilted Richardson variety} $\mathcal{T}_{u,v}$ and the \emph{open tilted Richardson variety} $\mathcal{T}_{u,v}^{\circ}$ which reduce to the well-known (open) Richardson variety if $u\leq v$ in the Bruhat order. 

\

\textbf{Main result 2:} definitions of the (open) tilted Richardson variety using
\begin{enumerate}
\item rank conditions (\Cref{def:main});
\item cyclically rotated Richardson varieties in the Grassmannian (\Cref{thm:altdefT});
\item multi-Pl\"ucker coordinates (\Cref{thm:altdefTplucker}).
\end{enumerate}

The (open) tilted Richardson varieties are our central geometric objects of study. Their geometric properties resemble those of Richardson varieties. However, since an analogue of Schubert varieties do not exist in our setting, most of the analysis requires new insights and techniques. 

\

\textbf{Main result 3:} geometric properties of the (open) tilted Richardson variety:
\begin{enumerate}
\item a stratification $\mathcal{T}_{u,v}=\sqcup_{[x,y]\subset[u,v]}\mathcal{T}_{x,y}^{\circ}$ that relates tilted Bruhat order (\Cref{thm:Tunion});
\item $\dim\mathcal{T}_{u,v}=\dim\mathcal{T}_{u,v}^{\circ}=\ell(u,v)$ in the quantum Bruhat graph (\Cref{thm:uv});
\item the closure relation $\overline{\mathcal{T}_{u,v}^{\circ}}=\mathcal{T}_{u,v}$ (\Cref{thm:closure}).
\end{enumerate}

\

In a sequel, we connect tilted Richardson varieties with curve neighborhoods $\Gamma_{d}(X_u,X^v)$. These are subvarieties of $\mathrm{Fl}_n$ introduced by Buch-Chaput-Mihalcea-Perrin in \cite{curveneighborhood} (see also \cite{BCMP}, \cite{BM15}, \cite{LiMihalcea} and references therein). They are closely related to computations in $QH^{*}(\mathrm{Fl}_n)$ and $QK^*(\mathrm{Fl}_n)$.

Our paper is organized as follows. In \Cref{sec:prelim}, we introduce background information on quantum Bruhat graphs, tilted Bruhat orders, root systems, Grassmannians, flag varieties, Richardson varieties and Pl\"ucker coordinates. In Section~\ref{sec:graph}, \ref{sec:tilted-Richardson-definition} and \ref{sec:tilted-Richardson-geometry}, we establish each aforementioned main results respectively on the combinatorics of quantum Bruhat graphs, definitions of the (open) tilted Richardson varieties, and their geometry.
\section{Preliminaries}\label{sec:prelim}
Let $S_n$ be the symmetric group on $n$ elements, generated by the \emph{simple transpositions} $S=\{s_i:=(i\ i{+}1)\text{ for }i=1,\ldots,n-1\}$. We typically write a permutation $w$ using its \emph{one-line} notation $w(1)w(2)\cdots w(n)$ or simplified as $w_1w_2\cdots w_n$. For $w\in S_n$, let $\ell(w)$ be the \emph{Coxeter length} of $w$, which is the smallest $\ell$ such that $w=s_{i_1}\cdots s_{i_{\ell}}$ is a product of $\ell$ simple transpositions. Such an expression is called a \emph{reduced word} of $w$. Let $R(w)$ be the set of reduced words of $w$. Let $I(w)=\{(i,j)\:|\: i<j,w_i>w_j\}$ be the \emph{inversion set} of $w$. It is well-known that $\ell(w)$ equals the number of inversions of $w$. Let $T=\{t_{ij}:=(i\ j)\:|\:1\leq i<j\leq n\}$ be the set of \emph{transpositions}, or equivalently, conjugates of the simple transpositions $S$. 
\subsection{The quantum Bruhat graph and the tilted Bruhat orders}
\begin{defin}\label{def:quantum-bruhat-graph}
The \emph{quantum Bruhat graph} $\Gamma_n$ is a weighted directed graph on $S_n$ with the following two types of edges:
\[
\begin{cases}
w\rightarrow wt_{ij}\text{ of weight }1&\text{ if }\ell(wt_{ij})=\ell(w)+1,\\
w\rightarrow wt_{ij}\text{ of weight }q_{ij}:=q_iq_{i+1}\cdots q_{j-1}&\text{ if }\ell(wt_{ij})=\ell(w)+1-2(j-i),
\end{cases}
\]
where $1\leq i<j\leq n$. Write $\wt(w\rightarrow wt_{ij})\in\mathbb{Z}[q_1,\ldots,q_{n-1}]$ for the weight.
\end{defin}
In other words, $w\rightarrow wt_{ij}$ if $w_i<w_j$ and for every $i<k<j$, $w_k>w_j$ or $w(k)<w(i)$, which gives an edge in the Hasse diagram of the strong Bruhat order, or $w\rightarrow wt_{ij}$ if $w_i>w_j$ and for every $i<k<j$, $w_i>w_k>w_j$. We can unify these two cases together using cyclic intervals.
\begin{defin}\label{def:cyclicinterval}
For $a\neq b\in [n]$, the (open) \emph{cyclic interval} from $a$ to $b$ is $(a,b)_c:=\{a+1,\ldots,b-1\}$ if $a<b$, and $(a,b)_c:=\{b+1,\ldots,n\}\cup\{1,\ldots,a-1\}$ if $a>b$. We can similarly define cyclic intervals $[,]$, $(,]$ and $[,)$.
\end{defin}
Now $w\rightarrow wt_{ij}$ in the quantum Bruhat graph if and only if $w_k\in (w_i,w_j)_c$ for all $i<k<j$. It is also clear that $k\in(a,b)_c$ if and only if $k+1\in(a+1,b+1)_c$. Therefore, we have the following immediate observation on the cyclic symmetry of the quantum Bruhat graph.
\begin{lemma}\label{lem:cyclic-symmetry}
Let $\tau\in S_n$ be the long cycle $(123\cdots n)$. Then $w\mapsto \tau w$ is an automorphism of the unweighted quantum Bruhat graph.
\end{lemma}

The quantum Bruhat graph for $n=3$ is shown in Figure~\ref{fig:quantum-Bruhat-graph-n=3}.
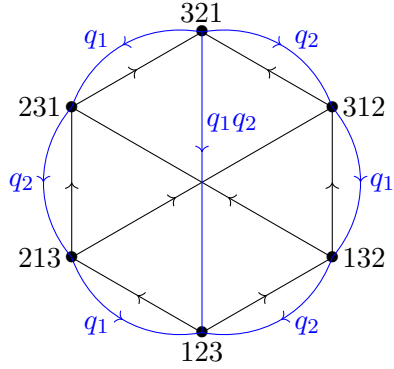
\begin{figure}[h!]
\centering
\begin{tikzpicture}[scale=1.0]
\coordinate (a) at (0,-2);
\coordinate (b) at (1.73205080757,-1);
\coordinate (c) at (1.73205080757,1);
\coordinate (d) at (0,2);
\coordinate (e) at (-1.73205080757,1);
\coordinate (f) at (-1.73205080757,-1);
\node at (a) {$\bullet$};
\node[below] at (a) {$123$};
\node at (b) {$\bullet$};
\node[right] at (b) {$132$};
\node at (c) {$\bullet$};
\node[right] at (c) {$312$};
\node at (d) {$\bullet$};
\node[above] at (d) {$321$};
\node at (e) {$\bullet$};
\node[left] at (e) {$231$};
\node at (f) {$\bullet$};
\node[left] at (f) {$213$};
\draw[postaction={decorate,decoration={markings,mark=at position 0.5 with {\arrow{>}}}}] (a) -- (b);
\draw[postaction={decorate,decoration={markings,mark=at position 0.5 with {\arrow{>}}}}] (a) -- (f);
\draw[postaction={decorate,decoration={markings,mark=at position 0.5 with {\arrow{>}}}}] (b) -- (c);
\draw[postaction={decorate,decoration={markings,mark=at position 0.4 with {\arrow{>}}}}] (b) -- (e);
\draw[postaction={decorate,decoration={markings,mark=at position 0.4 with {\arrow{>}}}}] (f) -- (c);
\draw[postaction={decorate,decoration={markings,mark=at position 0.5 with {\arrow{>}}}}] (f) -- (e);
\draw[postaction={decorate,decoration={markings,mark=at position 0.5 with {\arrow{>}}}}] (e) -- (d);
\draw[postaction={decorate,decoration={markings,mark=at position 0.5 with {\arrow{>}}}}] (c) -- (d);
\draw[blue, postaction={decorate,decoration={markings,mark=at position 0.5 with {\arrow{>}}}}] (b) to[bend left=40] (a);
\node[blue] at (1.4,-1.9) {$q_2$};
\draw[blue, postaction={decorate,decoration={markings,mark=at position 0.5 with {\arrow{>}}}}] (c) to[bend left=40] (b);
\node[blue] at (2.4,0) {$q_1$};
\draw[blue, postaction={decorate,decoration={markings,mark=at position 0.5 with {\arrow{>}}}}] (d) to[bend left=40] (c);
\node[blue] at (1.4,1.9) {$q_2$};
\draw[blue, postaction={decorate,decoration={markings,mark=at position 0.5 with {\arrow{>}}}}] (d) to[bend right=40] (e);
\node[blue] at (-1.4,1.9) {$q_1$};
\draw[blue, postaction={decorate,decoration={markings,mark=at position 0.5 with {\arrow{>}}}}] (e) to[bend right=40] (f);
\node[blue] at (-2.4,0) {$q_2$};
\draw[blue, postaction={decorate,decoration={markings,mark=at position 0.5 with {\arrow{>}}}}] (f) to[bend right=40] (a);
\node[blue] at (-1.4,-1.9) {$q_1$};
\draw[blue, postaction={decorate,decoration={markings,mark=at position 0.4 with {\arrow{>}}}}] (d) to (a);
\node[blue] at (0.4,0.8) {$q_1q_2$};
\end{tikzpicture}
\caption{Quantum Bruhat graph $\Gamma_3$ (unlabeled edges have weight $1$)}
\label{fig:quantum-Bruhat-graph-n=3}
\end{figure}

For a directed path $P:w^{(0)}\rightarrow w^{(1)}\rightarrow\cdots\rightarrow w^{(k)}$ in $\Gamma_n$, we say that $P$ has length $k$, with weight $\prod_{i=1}^k\wt(w^{(i-1)}\rightarrow w^{(i)})$. For $u,v\in S_n$, let $\ell(u,v)$ be the length of a shortest path from $u$ to $v$. Note that it is clear that for any $u,v\in S_n$, there exists a directed path from $u$ to $v$, as one can always go through the identity. Postnikov \cite{Postnikov-quantum-Bruhat-graph} established some nice properties regarding weights of shortest paths. Here, we write $q^{\alpha}$ for $q_1^{\alpha_1}\cdots q_{n-1}^{\alpha_{n-1}}$. 
\begin{lemma}[\cite{Postnikov-quantum-Bruhat-graph}]\label{lem:shortest-length-equals-weight}
For any $u,v\in S_n$, all shortest paths from $u$ to $v$ have the same weight $q^{d(u,v)}$. Moreover, the weight of any path from $u$ to $v$ is divisible by $q^{d(u,v)}$. In addition, if a path from $u$ to $v$ has weight $q^{d(u,v)}$, it must be a shortest path. 
\end{lemma}
Lemma~\ref{lem:shortest-length-equals-weight} is essentially saying that shortest length paths are precisely minimal weight paths. We remark that the last sentence from Lemma~\ref{lem:shortest-length-equals-weight} is not explicitly written done in \cite{Postnikov-quantum-Bruhat-graph}, but follows directly from the proof of Lemma~1 of \cite{Postnikov-quantum-Bruhat-graph}, with much of the main content established in Lemma~6.7 of \cite{BFP-tilted-Bruhat}. From now on, write $q^{d(u,v)}=q_1^{d_1(u,v)}q_2^{d_2(u,v)}\cdots q_{n-1}^{d_{n-1}(u,v)}$ for the minimal weight from $u$ to $v$, where $d(u,v)$ is a $(n-1)$-tuple.
\begin{ex}
Let $u=321$ and $v=213$. There are two shortest paths from $u$ to $v$ of length $2$: $321\rightarrow 231\rightarrow 213$ and $321\rightarrow 123\rightarrow 213$. We see that both paths have the same weight $q_1q_2$ so $d(321,213)=(1,1)$. It is straightfroward to check that any other paths from $u$ to $v$ have weight divisible by (and not equal to) $q_1q_2$. 
\end{ex}

\begin{defin}[tilted Bruhat order \cite{BFP-tilted-Bruhat}]\label{def:tilted-Bruhat-order}
For $u\in S_n$, define the \emph{tilted Bruhat order} $D_u$ to be the graded partial order on $S_n$ such that $w\leq_u v$ if
\begin{equation}\label{eqn:tilted-order-def}
    \ell(u,w)+\ell(w,v) = \ell(u,v).
\end{equation}
Equivalently, $w\leq_{u} v$ if there is a shortest path in the quantum Bruhat graph from $u$ to $v$ that passes through $w$.
\end{defin}

\begin{remark}
    A special case of the tilted Bruhat order is the strong Bruhat order. This is obtained by setting $u = id$.
\end{remark}
The tilted Bruhat order $D_{132}$ is shown in Figure~\ref{fig:tilted-132}.
\begin{figure}[h!]
\centering
\begin{tikzpicture}[scale=1.0]
\node at (0,0) {$\bullet$};
\node[below] at (0,0) {$132$};
\node at (-1.6,1) {$\bullet$};
\node[left] at (-1.6,1) {$123$};
\node at (0,1) {$\bullet$};
\node[above] at (0,1) {$231$};
\node at (1.6,1) {$\bullet$};
\node[right] at (1.6,1) {$312$};
\node at (-0.8,2) {$\bullet$};
\node[above] at (-0.8,2) {$213$};
\node at (0.8,2) {$\bullet$};
\node[above] at (0.8,2) {$321$};
\draw(0,0)--(-1.6,1)--(-0.8,2)--(0,1)--(0,0);
\draw(0,0)--(1.6,1)--(0.8,2)--(0,1);
\end{tikzpicture}
\caption{The tilted Bruhat order $D_{132}$}
\label{fig:tilted-132}
\end{figure}
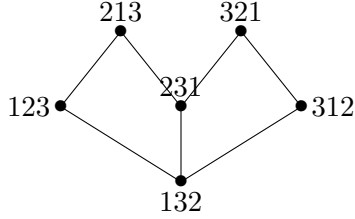

\begin{defin}
    For $w\leq_u v$, define the \emph{tilted Bruhat interval} to be 
    \[[w,v]_u = \{x\in S_n:w\leq_u x \leq_u v\}.\]
\end{defin}
\begin{remark}
    It follows from \eqref{eqn:tilted-order-def} that $[w,v]_u = [w,v]_{u'}$ as long as $w\leq_u v$ and $w\leq_{u'} v$. Since we always have $w\leq_w v$, we will omit the subscript and write $[w,v]$ instead of $[w,v]_u$ in the remaining part of this paper.
\end{remark}

\subsection{Root system and reflection ordering}
The root system of type $A_{n-1}$ consists of $\Phi=\{e_i-e_j\:|\: 1\leq i,\neq j\leq n\}$ with positive roots $\Phi^+=\{e_i-e_j\:|\: 1\leq i<j\leq n\}$ and simple roots $\Delta=\{\alpha_i:=e_i-e_{i+1}\:|\: 1\leq i\leq n-1\}$. It's corresponding Weyl group is identified with the symmetric group $S_n$, while the reflection across the hyperplane normal to $e_i-e_j$ is identified with $t_{ij}=(i\ j)$. 
\begin{defin}
An ordering of $\Phi^+$ (or equivalently, of the reflections $T$) as $\gamma_1,\ldots,\gamma_{\ell}$ is called a \emph{reflection ordering} if $e_i-e_k$ appears (not necessarily consecutively) in the middle of $e_i-e_j$ and $e_j-e_k$ for all $i<j<k$.
\end{defin}
The following lemma is very classical. See for example Proposition 3 of \cite{bjorner1984orderings}.
\begin{lemma}
Reflection orderings are in bijection with reduced words of the longest permutation $w_0=n\ n{-}1\cdots 21$. Given $a=a_1\cdots a_{\ell}\in R(w_0)$, its corresponding reflection ordering $\gamma_1,\ldots,\gamma_{\ell}$ is constructed via $\gamma_k=s_{a_1}\cdots s_{a_{j-1}}\alpha_{a_j}$ for $k=1,\ldots,\ell={n\choose 2}$.
\end{lemma}
\begin{ex}
Consider the reduced word $4321=s_3s_1s_2s_1s_3s_2$ which corresponds to the reflection ordering written on top of the arrows:
\[
\begin{tikzcd}
1234 \arrow[r, "e_3-e_4"]& 1243 \arrow[r, "e_1-e_2"] & 2143 \arrow[r, "e_1-e_4"] & 2413 \arrow[r, "e_2-e_4"] & 4213 \arrow[r, "e_1-e_3"] & 4231 \arrow[r, "e_2-e_3"] & 4321.
\end{tikzcd}
\]
\end{ex}

\subsection{Grassmaniann, Flag variety, and Pl\"ucker coordinates}

Define the \emph{Grassmaniann} $\gr(k,n)$ to be the space of $k$-dimensional subspaces $V\subseteq\mathbb{C}^n$. The \emph{Pl\"ucker embedding} is a closed embedding $\iota: \gr(k,n) \hookrightarrow \mathbb{P}(\Lambda^k(\mathbb{C}^n)))$ that sends the subspace $V$ with basis $\vec{v}_1,\vec{v}_2 ,\dots,\vec{v}_k$ to $[\vec{v}_1 \wedge \vec{v}_2 \wedge \dots \wedge \vec{v}_k]$, the projective equivalence class of  $\vec{v}_1 \wedge \vec{v}_2 \wedge \dots \wedge \vec{v}_k$ in $\Lambda^k(\mathbb{C}^n)$.
For any $i_1,\dots,i_k\in [n]$, let $P_{i_1,\dots,i_k}(V)$ be the \emph{Pl\"ucker coordinate} of $V$. For any permutation $\sigma\in S_k$,
\[P_{i_1,\dots,i_k}(V) = (-1)^{\ell(\sigma)} P_{\sigma(i_1),\dots,\sigma(i_k)}(V).\]
In particular, this means $P_{i_1,\dots,i_k}(V) = 0$ if $i_a = i_b$ for some $a\neq b\in [k]$.
For $I\in {[n]\choose k}$, set $P_I := P_{i_1,\dots,i_k}$ where $I = \{i_1<\dots<i_k\}$.
For any permutation $w\in S_n$, define 
\[P_w=:\prod_{k=1}^{n-1}P_{w[k]}\]
where $w[k]:=\{w_1,w_2,\dots,w_k\}\in {[n]\choose k}$ for $k\in [n-1]$. In particular, $P_{w[k]}(w) \in \{-1,1\}$. 
Set
$P_{I+i} := P_{i_1,\dots,i_k,i}$ and $P_{I-i_j} := (-1)^{k-j}P_{i_1,\dots,\widehat{i_j},\dots,i_k}$ for all $j\in [k]$. More generally, for $J = \{i_{j_1}<\dots<i_{j_r}\}\subset I$, set $P_{I-J}:= P_{((I-i_{j_r})-\dots)-i_{j_1}}$.
The variety $\gr(k,n)$ is the zero locus of the ideal generated by the following \emph{Pl\"ucker relations}.

\begin{defin}
    For any $I\in {[n]\choose k+1}, J\in {[n]\choose k-1}$, the \emph{Pl\"ucker relation} associated to $I,J$ is:
    \[\sum_{t\in I} P_{I-t}P_{J+t} = 0.\]
\end{defin}

Let $G = {GL}_n(\mathbb{C})$ and $B,B_{-}\subset G$ be the Borel and opposite Borel subgroup of $G$ consisting of invertible upper and lower triangular matrices respectively. Let $T=B\cap B_{-}$ be the maximal torus of diagonal matrices in $G$. 

The \emph{complete flag variety} is defined to be $\mathrm{Fl}_n = G/B$. Fixing a basis of $\mathbb{C}^n$, we can identify a point $gB\in G/B$ with a \emph{flag} $F_{\bullet} = 0=F_0\subsetneq F_1 \subsetneq F_2 \subsetneq \cdots \subsetneq F_{n-1}\subsetneq F_n=\mathbb{C}^n$ where $F_k\in\gr(k,n)$ is the span of the first $k$ column vectors of any $n\times n$ matrix representative $M_F\in gB$.
We define the \emph{multi-Pl\"ucker embedding} to be the composition \[\mathrm{Fl}_n \hookrightarrow \prod_{k = 0}^n \gr(k,n)\hookrightarrow \prod_{k = 0}^n \mathbb{P}(\Lambda^k (\mathbb{C}^n))\] that sends $F_\bullet$ to $(\iota(F_0), \iota(F_1),\dots,\iota(F_n))$. For any $k$ and $I\in {[n]\choose k}$, define the Pl\"ucker coordinate $P_I(F_\bullet)$ to be $P_I(F_k)$. It is also the minor of $M_F$ in the rows indexed by $I$ in the first $k$ columns. Similar to the Grassmaniann, the complete flag variety is the zero locus of the following \emph{incidence Pl\"ucker relations}:
\begin{defin}[\cite{youngtableaux}]
    For any $I\in {[n]\choose r}, J\in {[n]\choose s}$ with $1\leq s\leq r < n$ fix $A\subseteq J$, the \emph{incidence Pl\"ucker relation} is:
    \begin{equation}\label{eqn:fultonplucker}
        P_IP_J = \sum_{B\subseteq I, |B| = |A|}P_{(I- B)+ A}P_{(J- A)+ B}.
    \end{equation}
\end{defin}
There are two special cases of Equation~\eqref{eqn:fultonplucker} that will be of particular importance to us. We now rewrite the incidence Pl\"ucker relations in these cases for convenience.

\noindent\textbf{Case \RNum{1}.}($A = J$ and $|I| = |J|+1$): For $I\in {[n]\choose k}, J \in {[n]\choose k-1}$ where $1< k<n$,
\begin{equation}\label{eqn:incidenceplucker1}
    P_IP_J = \sum_{i\in I}P_{I-i}P_{J+i}
\end{equation}

\noindent\textbf{Case \RNum{2}.}($|A| = 1$ and $r-s\geq 2$): 
For $I\in {[n]\choose r}, J\in {[n]\choose s}$ with $1\leq s\leq r<n$,
\begin{equation}\label{eqn:incidenceplucker2}
     \sum_{i\in I} P_{I-i}P_{J+i} = 0.
 \end{equation}
Equivalently, 
\begin{equation}\label{eqn:incidenceplucker3}
    P_{I}P_{J+j} = -\sum_{i\in I}P_{I-i+j}P_{J+i}.
\end{equation}

\subsection{Schubert and Richardson varieties in $\mathrm{Fl}_n$}


Given a permutation $w\in S_n$, we also view $w$ as the permutation matrix with $1$ at $(w(i),i)$ and $0$ everywhere else\footnote{We note that in some literature, this is the permutation matrix that corresponds to $w^{-1}$ instead of $w$.}. The set of $T$-fixed points on $\mathrm{Fl}_n$ is $\{e_w = wB/B: w\in S_n\}$. Define the \emph{Schubert cell} $X_{w}^\circ = Be_w$ to be the Borel orbit of $e_w$ and the \emph{Schubert variety} $X_{w} = \overline{X_w^\circ}$ to be the closure of the Schubert cell. Similarly, define the \emph{opposite Schubert cell} $\Omega_{w}^\circ = B_{-}e_w$ and the \emph{opposite Schubert variety} $\Omega_w = \overline{\Omega_w^\circ}$. Here $X_{w}^\circ\cong \mathbb{C}^{\ell(w)}$ and $\Omega_w^\circ \cong \mathbb{C}^{{n\choose 2}-\ell(w)}$. 

We will make use of the following equivalent definition of (opposite) Schubert cells and Schubert varieties. For $S$ any subset of $[n]$, let $\mathrm{Proj}_S:\mathbb{C}^{n}\twoheadrightarrow \mathbb{C}^{|S|}$ be the projection onto the coordinates with indices in $S$. 
The Schubert cell is
\[X_w^{\circ} = \{F_{\bullet}\in \mathrm{Fl}_n: \dim(\mathrm{Proj}_{[n-j+1,n]}(F_i)) = \#\{w([i])\cap [n-j+1,n]\} \text{ for all }i,j\in [n]\},\]
and the Schubert variety is 
\[X_w = \{F_{\bullet}\in \mathrm{Fl}_n: \dim(\mathrm{Proj}_{[n-j+1,n]}(F_i)) \leq \#\{w([i])\cap [n-j+1,n]\} \text{ for all }i,j\in [n]\}.\]
Similarly, the opposite Schubert cell is
\[\Omega_w^{\circ} = \{F_{\bullet}\in \mathrm{Fl}_n: \dim(\mathrm{Proj}_{[j]}(F_i)) = \#\{w([i])\cap [j]\} \text{ for all }i,j\in [n]\},\]
and the opposite Schubert variety is
\[\Omega_w = \{F_{\bullet}\in \mathrm{Fl}_n: \dim(\mathrm{Proj}_{[j]}(F_i)) \leq \#\{w([i])\cap [j]\} \text{ for all }i,j\in [n]\}.\]

The (opposite) Schubert varieties possess a stratification by (opposite) Schubert cells:
\[X_{w} = \bigsqcup_{u\leq w}X_u^\circ\ \text{and } \Omega_w = \bigsqcup_{w\leq v}\Omega_v^\circ,\]
where ``$\leq$" denotes the strong Bruhat order on $S_n$.

It is perhaps easier to visualize Schubert and opposite Schubert varieties in the following way. For any $i,j\in [n]$ and any $M\in \mathrm{Mat}_{n\times n}$, define $r_{i,j}^{SW}(M)$ to be the rank of the submatrix of $M$ obtained by taking the bottom $n-i+1$ rows and left $j$ columns. Define similarly $r_{i,j}^{NW}(M)$ to be the rank of the submatrix obtained by taking the top $i$ rows and left $j$ columns. Let $M_F\in G$ be an matrix representative of $F_\bullet$ such that $F_i$ is the span of the first $i$ column vectors of $M_F$. Then
\begin{align}\label{eqn:Schubertrank}
    \begin{split}
        F_\bullet \in X_{w} &\iff r^{SW}_{i,j}(M_F)\leq r^{SW}_{i,j}(w)\\
    F_\bullet \in \Omega_{w} &\iff r^{NW}_{i,j}(M_F)\leq r^{NW}_{i,j}(w)
    \end{split},
\end{align}
and $F_\bullet$ lies in the (opposite) Schubert cell if \eqref{eqn:Schubertrank} holds after replacing ``$\leq$" with ``$=$".

Define the \emph{open Richardson variety} to be $\mathcal{R}^\circ_{u,v} = X^\circ_{v}\cap \Omega^\circ_{u}$ and the \emph{Richardson variety} to be $\mathcal{R}_{u,v} = X_{v}\cap \Omega_u$. In particular, $\mathcal{R}^\circ_{u,v}$ and $\mathcal{R}_{u,v}$ are non-empty if and only if $u\leq v$, in which case we have
\[\dim(\mathcal{R}_{u,v}) = \dim(\mathcal{R}_{u,v}^\circ) = \ell(v)-\ell(u) = \ell(u,v).\]
Similar to Schubert varieties, we also have
\begin{equation}\label{eqn:Richardson}
    \mathcal{R}_{u,v} = \overline{\mathcal{R}_{u,v}^{\circ}}\ \text{and }\mathcal{R}_{u,v} = \bigsqcup_{[x,y]\subseteq [u,v]}\mathcal{R}_{x,y}^\circ,
\end{equation}
where the disjoint union is taken over all Bruhat intervals $[x,y]$ contained in $[u,v]$.

Moreover, as subvarieties of $\mathrm{Fl_n}$, each $X_u,\Omega_u$ and $\mathcal{R}_{u,v}$ can be  defined by vanishing of Pl\"ucker coordinates:
\begin{align}\label{eqn:SchubPlucker}
        X_u &= \{F_\bullet\in \mathrm{Fl}_n:   P_w(F_\bullet) = 0 \text{ for }w \nleq u\} & X_u^\circ &= X_u\cap \{P_u \neq 0\}\nonumber\\
        \Omega_u &= \{F_\bullet\in \mathrm{Fl}_n:  P_w(F_\bullet) = 0 \text{ for }w \ngeq u\} & \Omega_u^\circ &= \Omega_u \cap \{P_u \neq 0\}\\
        \mathcal{R}_{u,v} &= \{F_\bullet\in \mathrm{Fl}_n:  P_w(F_\bullet) = 0 \text{ for } w \notin [u,v]\} & \mathcal{R}^\circ_{u,v} &= \mathcal{R}_{u,v} \cap \{P_uP_v \neq 0\}\nonumber.
\end{align}

\subsection{Schubert and Richardson varieties in Grassmannian} 
For each $I\in {[n]\choose k}$, define the \emph{Grassmannian Schubert variety} by
\[X_I = \{V\in \gr(k,n): \dim(\mathrm{Proj}_{[n-j+1,n]}) (V) \leq \#(I\cap [n-j+1,n])\text{ for all }j\in[n]\}\]
and the \emph{Grassmannian Schubert cell} by
\[X_I^\circ = \{V\in \gr(k,n): \dim(\mathrm{Proj}_{[n-j+1,n]}) (V) = \#(I\cap [n-j+1,n])\text{ for all }j\in[n]\}.\]
Define the \emph{Grassmannian opposite Schubert variety} and \emph{opposite Schubert cell} by
\[\Omega_I = \{V\in \gr(k,n): \dim(\mathrm{Proj}_{[j]}) (V) \leq \#(I\cap [j])\text{ for all }j\in[n]\}\]
and
\[\Omega_I^\circ = \{V\in \gr(k,n): \dim(\mathrm{Proj}_{[j]}) (V) = \#(I\cap [j])\text{ for all }j\in[n]\}.\]

Similar to Schubert varieties in $\mathrm{Fl}_n$, we have $X_I=\overline{X_I^\circ}$ and $\Omega_I=\overline{\Omega_I^\circ}$. For any two $k$-element subsets $I,J\subset [n]$ where $I = \{i_1<i_2<\dots<i_k\}$ and $J = \{j_1<j_2<\dots<j_k\}$, we say $I\leq J$ in the \emph{Gale order} if $i_r\leq j_r$ for all $r\in [k]$. Then the (opposite) Grassmaniann Schubert varieties are disjoint union of (opposite) Grassmaniann Schubert cells:
\[X_{I} = \bigsqcup_{J\leq I}X_J^\circ\ \text{and } \Omega_J = \bigsqcup_{I\leq J}\Omega_J^\circ.\]

When $I\leq J$, define the \emph{open Grassmaniann Richardson variety} $\mathcal{R}_{I,J}^\circ = X_{J}^\circ \cap \Omega_{I}^\circ$ and its closure the \emph{Grassmaniann Richardson variety} $\mathcal{R}_{I,J} = X_{J}\cap \Omega_I$. We similarly have
\begin{equation}\label{eqn:grRichardson}
    \mathcal{R}_{I,J} = \overline{\mathcal{R}_{I,J}^{\circ}}\ \text{and }\mathcal{R}_{I,J} = \bigsqcup_{[I',J']\subseteq [I,J]}\mathcal{R}_{I',J'}^\circ.
\end{equation}

One can also define $X_I,\Omega_I$ and $\mathcal{R}_{I,J}$ by vanishing of Pl\"ucker coordinates:
\begin{align}\label{eqn:grSchubPlucker}
        X_I &= \{V\in \gr(k,n):   P_K(V) = 0\text{ for }K \nleq I\} & X_I^\circ &= X_I\cap \{P_I \neq 0\}\nonumber\\
        \Omega_I &= \{V\in \gr(k,n):   P_K(V) = 0\text{ for }K \ngeq I\} & \Omega_I^\circ &= \Omega_I \cap \{P_I \neq 0\}\\
        \mathcal{R}_{I,J} &= \{V\in \gr(k,n):   P_K(V) = 0\text{ for }K \notin [I,J]\} & \mathcal{R}^\circ_{I,J} &= \mathcal{R}_{I,J} \cap \{P_IP_J \neq 0\}\nonumber.
\end{align}

\section{Combinatorics of the quantum Bruhat graph}\label{sec:graph}
\subsection{Minimal weights in the quantum Bruhat graph}
Our first theorem provides an explicit formula for the minimal weight $q^{d(u,v)}$ from any pair of permutation $u$ to $v$ in $S_n$. 
\begin{defin}\label{def:up-down-path}
For $A,B\subset[n]$ with $|A|=|B|$, we construct a lattice path $P(A,B)$ starting at $(0,0)$ and ending at $(n,0)$ with $n$ steps as follows. For each $i=1,\ldots,n$, the $i^{th}$ step is
\begin{itemize}
\item an upstep $(1,1)$ if $i\in A$ and $i\notin B$,
\item a downstep $(1,-1)$ if $i\notin A$ and $i\in B$,
\item a horizontal step $(1,0)$ if $i\in A\cap B$ or $i\notin A\cup B$. 
\end{itemize}
Define its \emph{depth} to be the largest number $y\geq0$ such that this lattice path passes through $(x,-y)$ for some $x$, denoted $\depth(A,B)$. 
\end{defin}
\begin{ex}\label{ex:lattice-path-PAB}
Let $n=7$ with $A=\{3,4,6,7\}$ and $B=\{1,2,3,5\}$. The lattice path $P(A,B)$ is shown in Figure~\ref{fig:lattice-path-PAB} with $\depth(A,B)=2$.
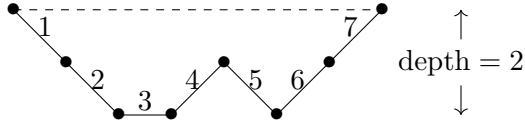
\begin{figure}[h!]
\centering
\begin{tikzpicture}[scale=0.7]
\node at (0,0) {$\bullet$};
\node at (1,-1) {$\bullet$};
\node at (2,-2) {$\bullet$};
\node at (3,-2) {$\bullet$};
\node at (4,-1) {$\bullet$};
\node at (5,-2) {$\bullet$};
\node at (6,-1) {$\bullet$};
\node at (7,0) {$\bullet$};
\draw(0,0)--(2,-2)--(3,-2)--(4,-1)--(5,-2)--(7,0);
\draw[dashed](0,0)--(7,0);
\node (d) at (8.5,-1) {$\mathrm{depth}=2$};
\draw[->] (d)--(8.5,0);
\draw[->] (d)--(8.5,-2);
\node[above] at (0.6,-0.7) {$1$};
\node[above] at (1.6,-1.7) {$2$};
\node[above] at (2.5,-2.1) {$3$};
\node[above] at (3.4,-1.7) {$4$};
\node[above] at (4.6,-1.7) {$5$};
\node[above] at (5.4,-1.7) {$6$};
\node[above] at (6.4,-0.7) {$7$};
\end{tikzpicture}
\caption{The lattice path $P(A,B)$}
\label{fig:lattice-path-PAB}
\end{figure}
\end{ex}

\begin{theorem}\label{thm:weight-distance}
Let $u,v\in S_n$. All shortest paths from $u$ to $v$ have weight $q^{d(u,v)}=q_1^{d_1(u,v)}\cdots q_{n-1}^{d_{n-1}(u,v)}$ where $d_k(u,v)=\depth(u[k],v[k])$. Here, $w[k]:=\{w_1,\ldots,w_k\}$.
\end{theorem}
\begin{ex}
Consider $u=7364152$ and $v=2513746$ in $S_7$. We need to figure out $d_k(u,v)$ for $k=1,\ldots,n-1$. For $k=1$, $u[1]=\{7\}$ and $v[1]=\{2\}$ so we need to construct a lattice path with an upstep at position $7$ and a downstep at position $2$. This lattice path has depth $1$, meaning that $d_1(u,v)=1$. We continue this procedure for all $k$'s. For example, at $k=4$, the lattice path $P(u[4],v[4])$ is discussed in Example~\ref{ex:lattice-path-PAB} with $d_4(u,v)=2$. In the end, we arrive at $q^{d(u,v)}=q_1q_2q_3^2q_4^2q_5q_6$.
\end{ex}

We make use of the following technical tool by Brenti-Fomin-Postnikov. In the quantum Bruhat graph $\Gamma_n$, label each directed edge $w\rightarrow wt_{ij}$ by $e_i-e_j$. The ``label" here does not have indications towards the weight. 
\begin{theorem}[\cite{BFP-tilted-Bruhat}]\label{thm:shortest-path-BFP}
Fix a reflection ordering $\tau=\tau_1,\ldots,\tau_{n\choose 2}$ of $\Phi^+$ and fix $u,v\in S_n$. Then there is a unique directed path from $u$ to $v$ in $\Gamma_n$ such that its sequence of labels is strictly increasing with respect to $\tau$. Moreover, this path has length $\ell(u,v)$.
\end{theorem}
Therefore, to prove Theorem~\ref{thm:weight-distance}, we pick a specific reflection ordering, and find its corresponding directed path and compute its weight, recalling that shortest paths are precisely minimal weight paths (Lemma~\ref{lem:shortest-length-equals-weight}). We need one more definition.

\begin{defin}
For $r\in [n]$, let $\leq_r$ be the \emph{shifted linear order} on $[n]$ given by \[r<_r r+1 <_r \cdots <_r n <_r 1 <_r \cdots <_r r-1.\] For $A = \{a_1 <_r \cdots <_r a_k\},B = \{b_1<_r \cdots <_r b_k\} \subset [n]$, define the \emph{shifted Gale order} $\leq_r$ as
\[A\leq_r B \iff a_i \leq_r b_i \text{ for all }i\in [k].\]
\end{defin}

\begin{proof}[Proof of Theorem~\ref{thm:weight-distance}]
Choose the reflection ordering $\tau=e_1-e_2,e_1-e_3,\ldots,e_1-e_n,e_2-e_3,\ldots e_2-e_{n},\ldots,e_{n-1}-e_n$ that starts with all positive roots involving $e_1$ and then $e_2$ and so on. This is a valid reflection order because for any $i<j<k$, $e_i-e_j$, $e_i-e_k$ and $e_j-e_k$ appear in this order. One can also check that this ordering corresponds to the reduced word $w_0=(s_1s_2\cdots s_{n-1})(s_1s_2\cdots s_{n-2})\cdots (s_1)$. 

Fix $u$ and $v$. Theorem~\ref{thm:shortest-path-BFP} says that there is a unique directed path from $u$ to $v$ in $\Gamma_n$ by $u\rightarrow ut_{i_1j_1}\rightarrow ut_{i_1j_1}t_{i_2j_2}\rightarrow\cdots\rightarrow ut_{i_1j_1}\cdots t_{i_\ell j_{\ell}}=v$ such that $e_{i_1}-e_{j_1},\ldots,e_{i_{\ell}}-e_{j_{\ell}}$ appear in the order of $\tau$. We describe this unique directed path explicitly. Consider the first $n-1$ roots $e_1-e_2,\ldots,e_1-e_{n}$. After choosing $e_1-e_{p_1},e_1-e_{p_2},\ldots,e_1-e_{p_m}$ such that we have a directed path $u\rightarrow ut_{1p_1}\rightarrow\cdots\rightarrow ut_{1p_1}\cdots t_{1p_m}=u'$, we must need $u'(1)=v(1)=v_1$ since the rest of positive roots do not involve the index $1$. At the same time, as long as $u'(1)=v(1)$, we can continue finding the path via induction on $n$.

We can choose $p_i$ to be the smallest index greater than $p_{i-1}$ (with the convention that $p_0=1$) such that $u(p_i)>_{v_1+1}u(p_{i-1})$ in the shifted linear order on $[n]$ where $v_1$ is declared the largest number. Such choices work because by construction, there is an edge from $ut_{1p_1}\cdots t_{1p_{i-1}}$ to $ut_{1p_1}\cdots t_{1p_i}$ in $\Gamma_n$ and we inevitably end up at $p_m=u^{-1}(v_1)$ since $v_1$ is the largest number in this linear order. Here is an example with $v_1=4$:
\[\begin{tabular}{c|c|c|c}
& permutation & root & weight \\ \hline
$u$ & $\underline{\textbf{6}}5\underline{\textbf{7}}913428$ & $e_1-e_3$ & 1  \\ \hline
 & $\underline{\textbf{7}}56\underline{\textbf{9}}13428$ & $e_1-e_4$ & 1 \\ \hline
 & $\underline{\textbf{9}}567\underline{\textbf{1}}3428$ & $e_1-e_5$ & $q_1q_2q_3q_4$ \\\hline
 & $\underline{\textbf{1}}5679\underline{\textbf{3}}428$ & $e_1-e_6$ & 1 \\\hline
 & $\underline{\textbf{3}}56791\underline{\textbf{4}}28$ & $e_1-e_7$ & 1 \\\hline
 & $\textbf{4}56791328$ & & \\
 & $\cdots$ & & \\\hline
$v$ & $\textbf{4}\rule{1.46cm}{0.15mm}$ & & \\\hline
\end{tabular}.\]

We now examine how the weight accumulated. The weight of this path $u\rightarrow\cdots\rightarrow u'$ equals $q_1q_2\cdots q_{p_k-1}$ if $u(p_k)<u(p_{k-1})$ and $u(p_k)>_{v_1+1}u(p_{k-1})$, and equals $1$ if no such $k$ exists, i.e. $u(1)<v(1)$. It's clear that at most one such $k$ exists which is intuitively at the place where the shifted linear order ``wraps around". Denote this weight by $\wt_1(u,v)$. Recall that the weight of a minimal weight path from $u$ to $v$ is denoted as $q^{d(u,v)}$. Also let $\depth(u,v)$ be a vector such that $\depth_j(u,v)=\depth(u[j],v[j])$ as in Definition~\ref{def:up-down-path}. 

Our goal is to show that $d(u,v)=\depth(u,v)$ and we use induction on $n$. By Theorem~\ref{thm:shortest-path-BFP}, $\wt_1(u,v)q^{d(u',v)}=q^{d(u,v)}$, and by induction hypothesis, $q^{d(u',v)}=q^{\depth(u',v)}$ since $u'(1)=v(1)$. It remains to prove that 
\begin{equation}\label{eqn:induction-on-depth}
\wt_1(u,v)q^{\depth(u',v)}=q^{\depth(u,v)}.
\end{equation}
We study Equation~\eqref{eqn:induction-on-depth} coordinate by coordinate.

\

\noindent\textbf{Case \RNum{1}:} ($u_1>v_1$). There exists a unique $k$ such that $u(p_k)<u(p_{k-1})$ and $u(p_k)>_{v_1+1}u(p_{k-1})$, with $\wt_1(u,v)=q_1q_2\cdots q_{p_k-1}$. In fact, all of $u(1),\ldots,u(p_{k}-1)$ are strictly greater than $v_1$ by construction. Now compare the exponent of $q^j$ on both sides of Equation~\eqref{eqn:induction-on-depth}. In other words, we will compare the two lattice paths $P(u[j],v[j])$ and $P(u'[j],v[j])$. 

For $j<p_k$, $u'[j]$ is obtained from $u[j]$ by deleting its largest element and replacing it by $v(1)$ which is smaller than all values in $u[j]$. Thus, the path $P(u[j],v[j])$ is strictly below the $x$-axis when the $x$-coordinate is $v(1)$, and is weakly above the $x$-axis when the $x$-coordinate is $\max u[j]$. The lattice path $P(u'[j],v[j])$ is obtained from $P(u[j],v[j])$ by moving the subpath from $(v_1,-)$ to $(\max u[j],-)$ one step up. As a result, $\depth(u[j],v[j])=\depth(u'[j],v[j])+1$ as desired. A cartoon for visualization for this scenario is seen in Figure~\ref{fig:path-depth-change-1}.
\begin{figure}[h!]
\centering
\begin{tikzpicture}[scale=0.55]
\node at (0,0) {$\bullet$};
\node at (1,-1) {$\bullet$};
\node at (2,-1) {$\bullet$};
\node at (3,-2) {$\bullet$};
\node[right] at (3,-2) {$(v_1,-)$};
\node at (6,2) {$\bullet$};
\node[above] at (6,2) {$(\max u[j],-)$};
\node at (7,2) {$\bullet$};
\node at (8,1) {$\bullet$};
\node at (9,0) {$\bullet$};
\draw(0,0)--(1,-1)--(2,-1)--(3,-2);
\draw(6,2)--(7,2)--(9,0);
\draw[dashed](0,0)--(9,0);
\coordinate (A) at (3,-2);
\coordinate (B) at (6,2);
\draw[red,thick] (A) .. controls (4,2) and (2,1) .. (B);
\def\a{12}
\node at (\a,0) {$\bullet$};
\node at (\a+1,-1) {$\bullet$};
\node at (\a+2,-1) {$\bullet$};
\node at (\a+3,-1) {$\bullet$};
\node[right] at (\a+3,-1) {$(v_1,-)$};
\node at (\a+6,3) {$\bullet$};
\node[left] at (\a+6,3) {$(\max u[j],-)$};
\node at (\a+7,2) {$\bullet$};
\node at (\a+8,1) {$\bullet$};
\node at (\a+9,0) {$\bullet$};
\draw(\a+0,0)--(\a+1,-1)--(\a+2,-1)--(\a+3,-1);
\draw(\a+6,3)--(\a+7,2)--(\a+9,0);
\draw[dashed](\a+0,0)--(\a+9,0);
\coordinate (A2) at (\a+3,-1);
\coordinate (B2) at (\a+6,3);
\draw[red,thick] (A2) .. controls (\a+4,3) and (\a+2,2) .. (B2);
\end{tikzpicture}
\caption{A comparison for $P(u[j],v[j])$ (left) and $P(u'[j],v[j])$ (right) for $j<p_k$ in case 1 of the proof of Theorem~\ref{thm:weight-distance} where the red curve stands for any lattice subpath}
\label{fig:path-depth-change-1}
\end{figure}
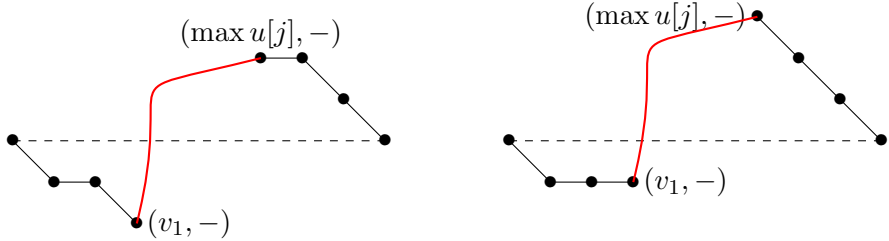

For $j\geq p_m$, $u'[j]=u[j]$ so $\depth(u'[j],v[j])=\depth(u[j],v[j])$ as desired.

For $p_k\leq j<p_m$, $v_1\notin u[j]$ and $u'[j]$ is obtained from $u[j]$ by deleting the largest number smaller than $v_1$, called $b$, and replacing it by $v_1$. Thus, $P(u'[j],v[j])$ and $P(u[j],v[j])$ only differ from step $b$ to step $v_1$. During this period,
\begin{align*}
P(u[j],v[j]):\ & (\nearrow\text{ or }\rightarrow)\ +\ (\text{a sequence of }\rightarrow\text{ and }\searrow\text{'s})\ +\ (\searrow), \\
P(u'[j],v[j]):\ & (\rightarrow\text{ or }\searrow)\ +\ (\text{a sequence of }\rightarrow\text{ and }\searrow\text{'s})\ +\ (\rightarrow).
\end{align*}
This local change does not modify the depth. So $\depth(u[j],v[j])=\depth(u'[j],v[j])$.

\

\noindent\textbf{Case \RNum{2}:} ($u_1\leq v_1$). Here $\wt_1(u,v)=1$ and we need $\depth(u'[j],v[j])=\depth(u[j],v[j])$ for all $j$. The sub-case for $j\geq p_m$ is exactly the same as the sub-case for $j\geq p_m$ in Case 1, and the sub-case for $j<p_m$ is exactly the same as the sub-case for $p_k\leq j<p_m$ in Case 1. 

The induction step goes through and we conclude that $d(u,v)=\depth(u,v)$.
\end{proof}

\subsection{Characterization of the tilted Bruhat order}
A crucial consequence of Theorem~\ref{thm:weight-distance} is a succinct description (Theorem~\ref{thm:tilted-criterion}) of the tilted Bruhat order defined by Brenti-Fomin-Postnikov. We now build up some intuition regarding the relationship between the shifted Gale order $\leq_r$ and the lattice path construction (Definition~\ref{def:up-down-path}).

\begin{lemma}\label{lemma:r-comparable}
For all $A,B \subset [n]$ with $|A| = |B|$, there exists $r\in [n]$ such that $A\leq_r B$. In fact, $A\leq_r B$ if and only if the lattice path $P(A,B)$ passes through $(r-1,-\depth(A,B))$.
\end{lemma}
\begin{proof}
Note that $A\leq B$ in the Gale order if and only if the lattice path $P(A,B)$ does not go below the $x$-axis, i.e. $P(A,B)$ is a Dyck path. Likewise, $A\leq_r B$ if and only if the lattice path $P(A+r-1,B+r-1)$ does not go below the $x$-axis, where $A+r-1=\{a+r-1\:|\: a\in A\}$ where the values are taken modulo $n$. Therefore, for general $A$ and $B$ with $|A|=|B|$, let the steps of $P(A,B)$ be $g_1,g_2,\ldots,g_n$ where each $g_i\in\{\rightarrow,\nearrow,\searrow\}$ based on Definition~\ref{def:up-down-path}. Then the lattice path $P(A+r-1,B+r-1)$ is then constructed via the steps $g_r,g_{r+1},\ldots,g_n,g_1,\ldots,g_{r-1}$ in this order. We now see that $P(A+r-1,B+r-1)$ does not go below the $x$-axis if and only if the lattice path $P(A,B)$ goes through its lowest point at $x$-coordinate equals $r-1$. 
\end{proof}

If $A\leq_rB$, let $[A,B]_r:=\{I:A\leq_rI\leq_rB\}$ denote the interval in shifted Gale order $\leq_r$. The following lemma explains the relation between different shifted Gale orders.

\begin{lemma}\label{lemma:both-r-comparable}
    For all $A,B\subset [n]$ and $|A|=|B|$, if there exists $r\neq r'\in [n]$ such that $A\leq_rB$ and $A\leq_{r'}B$, then the intervals $[A,B]_r=[A,B]_{r'}$. Moreover, for any $I\in[A,B]_r$,
    \[
        \#(A\cap [r,r')_c) = \#(I\cap [r,r')_c)=\#(B\cap [r,r')_c).
    \]
\end{lemma}
\begin{proof}
    Without loss of generality assume $r<r'$. Since $A\leq_r B$ and $A\leq_{r'} B$, by Lemma~\ref{lemma:r-comparable}, the lattice path $P(A,B)$ passes through $(r,-d)$ and $(r',-d)$, where $d=\depth(A,B)$. Since they have the same $y$-coordinate, there are the same number of $\nearrow$ and $\searrow$ between the two points and thus
    \[\#(A\cap [r,r')_c) = \#(B\cap [r,r')_c).\]
    For any $I\in[A,B]_r$, we have 
    \[I\geq_r A\implies \#(I\cap [r,r')_c)\leq\#(A\cap [r,r')_c),\] 
    \[I\leq_r B\implies\#(I\cap [r,r')_c)\geq\#(B\cap [r,r')_c).\]
    Therefore,
    \[\#(I\cap [r,r')_c) = \#(A\cap [r,r')_c) = \#(B\cap [r,r')_c).\]
    As a result, the lattice path $P(A,I)$ have the same $y$-coordinate at $x=r$ and $r'$. Since $A\leq_r I$, $P(A,I)$ reaches its lowest point at both $x$-coordinate $r$ and $r'$. Similarly, $P(I,B)$ reaches its lowest point at both $x$-coordinate $r$ and $r'$. By Lemma~\ref{lemma:r-comparable}, $A\leq_{r'}I\leq_{r'}B$. Therefore $[A,B]_{r}\subseteq[A,B]_{r'}$. The opposite direction follows by same reasoning.
\end{proof}

\begin{theorem}\label{thm:tilted-criterion}
For $u,v,w\in S_n$, the following are equivalent:
\begin{enumerate}
\item $w\leq_u v$;
\item $w\in [u,v]$;
\item for all sequence $\mathbf{a} = a_1\ldots a_{n-1}$ such that $\{u_1,\ldots,u_k\}\leq_{a_k} \{v_1,\ldots,v_k\}$ and for all $k\in [n-1]$, we have $\{u_1,\ldots,u_k\}\leq_{a_k}\{w_1,\ldots w_k\}\leq_{a_k} \{v_1,\ldots,v_k\}$;
\item there exists a sequence $\mathbf{a} = a_1\ldots a_{n-1}$ such that for all $k\in [n-1],\{u_1,\ldots,u_k\}\leq_{a_k}\{w_1,\ldots w_k\}\leq_{a_k} \{v_1,\ldots,v_k\}$.
\end{enumerate}
\end{theorem}
\begin{proof}
$(1)\iff (2)$ is clear by definition.
We will show that $(1)\implies (3)\implies (4)\implies (1)$.  $(3)\implies (4)$ is immediate given that Lemma~\ref{lemma:r-comparable} establishes the existence of $a_k$'s satisfying $\{u_1,\ldots,u_k\}\leq_{a_k}\{v_1,\ldots,v_k\}$. 

\

\noindent $(1)\implies (3):$ 
Recall that the long cycle $\tau=(123\cdots n)\in S_n$ is an automorphism of the unweighted quantum Bruhat graph (Lemma~\ref{lem:cyclic-symmetry}). Fix a sequence $\mathbf{a}=a_1\ldots,a_{n-1}$ such that $\{u_1,\ldots,u_k\}\leq_{a_k}\{v_1,\ldots,v_k\}$ for all $k$. As $w\leq_u v$, there is a shortest path from $u$ to $v$ that passes through $w$. As shortest length paths equal minimal weight paths (Lemma~\ref{lem:shortest-length-equals-weight}), we also have $d(u,v)=d(u,w)+d(w,v)$, where $d(u,v)$ is the exponent vector of the weight of any shortest path from $u$ to $v$. Since $\tau$ is an automorphism, the same logic also applies to $\tau^{1-a_k}w\leq_{\tau^{1-a_k}u}\tau^{1-a_k}v$ so \[d(\tau^{1-a_k}u,\tau^{1-a_k}v)=d(\tau^{1-a_k}u,\tau^{1-a_k}w)+d(\tau^{1-a_k}w,\tau^{1-a_k}v).\]
By Theorem~\ref{thm:weight-distance}, the $k^{th}$ coordinate of $d(\tau^{1-a_k}u,\tau^{1-a_k}v)$ is given by the depth of the lattice path $P(\tau^{1-a_k}u[k],\tau^{1-a_k}v[k])$ which is $0$ because $\{u_1,\ldots,u_k\}\leq_{a_k}\{v_1,\ldots,v_k\}$. As a result, the $k^{th}$ coordinate of both $d(\tau^{1-a_k}u,\tau^{1-a_k}w)$ and $d(\tau^{1-a_k}w,\tau^{1-a_k}v)$ must be $0$ as well, which translate to $\{u_1,\ldots,u_k\}\leq_{a_k}\{w_1,\ldots,w_k\}$ and $\{w_1,\ldots,w_k\}\leq_{a_k}\{v_1,\ldots,v_k\}$.

\

\noindent $(4)\implies (1):$ Fix such a sequence $\mathbf{a}$ and look at a coordinate $k$. Since $u[k]\leq_{a_k} w[k]\leq_{a_k} v[k]$, the three lattice paths $P(u[k],v[k])$, $P(u[k],w[k])$ and $P(w[k],v[k])$ all reach their lowest points at $x$-coordinate $a_k-1$ by Lemma~\ref{lemma:r-comparable}. Then by construction (\Cref{def:up-down-path}),
\[\depth(u[k],v[k])=|\{b\in v[k]: b<a_k\}|-|\{b\in u[k]: b<a_k\}|\geq 0.\]
This description allows us to conclude $\depth(u[k],v[k])=\depth(u[k],w[k])+\depth(w[k],v[k])$. Then $d(u,v)=d(u,w)+d(w,v)$ by Theorem~\ref{thm:weight-distance} which implies $w\leq_u v$ by Lemma~\ref{lem:shortest-length-equals-weight} and Definition~\ref{def:tilted-Bruhat-order}. 
\end{proof}

The following is immediate from Theorem~\ref{thm:tilted-criterion}:
\begin{cor}\label{cor:xy}
    Let $[x,y]\subset [u,v]$ be any subinterval. Let $\mathbf{a}$ be any sequence such that $u[k]\leq_{a_k} v[k]$ for all $k\in [n-1]$, then for all $k\in [n-1]$,
    \[x[k]\leq_{a_k} y[k].\]
\end{cor}
\section{Tilted Richardson varieties}\label{sec:tilted-Richardson-definition}
From this point on, we turn our attention to geometry. The goal of this section is to give three equivalent definitions to \emph{(open) tilted Richardson varieties} $\mathcal{T}_{u,v}^{\circ}$ and $\mathcal{T}_{u,v}$ for any pair of permutations $u,v\in S_n$. These varieties of interest will later be proved to provide a geometric meaning to the tilted Bruhat order (\Cref{thm:Tunion}). In particular, if $u\leq v$ in the Bruhat order, the (open) tilted Richardson varieties coincide with the usual (open) Richardson varieties $\mathcal{R}_{u,v}^{\circ}$ and $\mathcal{R}_{u,v}$. These three equivalent definitions use:
\begin{enumerate}
\item rank conditions (\Cref{def:main});
\item cyclically rotated Richardson varieties in the Grassmannian (\Cref{thm:altdefT});
\item multi-Pl\"ucker coordinates (\Cref{thm:altdefTplucker}).
\end{enumerate}
As an intermediate step, we first work with $\mathcal{T}^\circ_{u,v,\mathbf{a}}$ and $\mathcal{T}_{u,v,\mathbf{a}}$ and then show that they do not depend on the choice of sequence $\mathbf{a}$ (\Cref{cor:independent-a}). Here, given $u,v\in S_n$ and a sequence $\mathbf{a} = (a_1, \ldots, a_{n-1})$, we write $u\leq_{\mathbf{a}} v$ if $u[k]\leq_{a_k}v[k]$ for all $k\in [n-1]$.


When $a\neq b\in[n]$, recall the definition of cyclic intervals $[a,b]_c,[a,b)_c,(a,b]_c,(a,b)_c$ in Definition~\ref{def:cyclicinterval}. For convenience, we extend the definition by setting $[j,0]_c :=[j,n] =  \{j,j+1,\dots,n\}, [j,j]_c := \{j\}$ and $[j,j)_c = (j,j]_c := \emptyset$.

\subsection{The first definition: rank conditions}

Our first definition is motivated by our characterization of tilted Bruhat order (\Cref{thm:tilted-criterion}).
\begin{defin}\label{def:main}
Define the \emph{open tilted Richardson variety} with respect to $\mathbf{a}$ to be:
\begin{align}\label{eqn:tiltRichOpen}
    \mathcal{T}_{u,v,\mathbf{a}}^\circ = \left\{F_{\bullet} \in \mathrm{Fl}_n:
    \begin{array}{c}
         \dim (\mathrm{Proj}_{[a_i,j]_c}(F_i))= \#\{u[i]\cap [a_i,j]_c\},   \\
         \dim (\mathrm{Proj}_{[j,a_{i}-1]_c}(F_i))= \#\{v[i]\cap [j,a_i-1]_c\} 
    \end{array} 
    \forall i,j\in [n]
     \right\}.
\end{align}
Define the \emph{tilted Richardson variety} with respect to $\mathbf{a}$ to be:
\begin{align}\label{eqn:tiltRich}
    \mathcal{T}_{u,v,\mathbf{a}} = \left\{F_{\bullet} \in \mathrm{Fl}_n:
    \begin{array}{c}
         \dim (\mathrm{Proj}_{[a_i,j]_c}(F_i))\leq \#\{u[i]\cap [a_i,j]_c\},   \\
         \dim (\mathrm{Proj}_{[j,a_{i}-1]_c}(F_i))\leq \#\{v[i]\cap [j,a_i-1]_c\} 
    \end{array} 
    \forall i,j\in [n]
     \right\}.
\end{align}
\end{defin}

Indeed, for any 
$u\leq_{\mathbf{a}}v$, $\dim (\mathrm{Proj}_{[a_i,j]_c}(wB/B)) = \#\{w[i]\cap [a_o,j]_c\}$. Therefore
\[e_w\in \mathcal{T}_{u,v,\mathbf{a}} \iff w\in [u,v].\] 
In particular, $e_w\in \mathcal{T}_{u,v,\mathbf{a}}$ is independent of the choice of $\mathbf{a}$. We will see later (\Cref{cor:independent-a}) that in fact $\mathcal{T}_{u,v,\mathbf{a}}$ and $\mathcal{T}_{u,v,\mathbf{a}}^\circ$ are both independent of $\mathbf{a}$. 

It is easier to visualize tilted Richardson varieties as follows. For any $S\subseteq [n]$, $k\in [n]$ and $M\in \mathrm{Mat}_{n\times n}$, define $r_{S,k}(M)$ to be the rank of the submatrix of $M$ obtained by taking the rows in $S$ and the left $k$ columns. Let $M_F$ be a matrix representative of $F_\bullet$. Then
\begin{align}\label{eqn:Tuvrank}
    F_\bullet \in \mathcal{T}_{u,v,\mathbf{a}} \iff \begin{cases}
        \begin{split}
        r_{[a_i,j]_c,i}(M_F)&\leq r_{[a_i,j]_c,i}(u)\\
        r_{[j,a_{i}-1]_c,i}(M_F)&\leq r_{[j,a_{i}-1]_c,i}(v)
    \end{split}
    \end{cases},\text{ for all }i,j\in [n]
\end{align}
and $F_\bullet\in\mathcal{T}_{u,v,\mathbf{a}}^{\circ}$ if we replace ``$\leq$" with ``$=$" in \eqref{eqn:Tuvrank}.

\begin{ex}\label{ex:4321-3142}
Let $u = 4321$ and $v = 3142$. Then $u\leq_\mathbf{a} v$ for all choices of $a_1\in \{4\}, a_2\in \{2,3,4\}$ and $a_3 \in \{2\}$. See Figure~\ref{fig:4321-3142} for an illustration where $\star$ and $\bullet$ represent $u$ and $v$ respectively. The red horizontal line segment in column $k$ represent the cutoff of $[n]$ under $\leq_{a_k}$ for different choices of $a_k$, where $k\in\{1,2,3\}$.

Let $\mathbf{a} = (4,4,2)$. For $F_\bullet \in \mathcal{T}_{u,v,\mathbf{a}}$, there are $8$ rank conditions imposed on $F_2$ as in \eqref{eqn:tiltRich}. The condition $\dim(\mathrm{Proj}_{\{1,2,3\}}(F_2))\leq 2 = \#v[2]\cap\{1,2,3\}$ is interpreted as the rank of the shaded submatrix in Figure~\ref{fig:4321-3142} being at most $2$, the number of $\bullet$ in said region. 
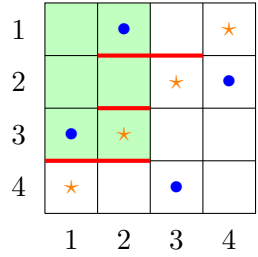
\begin{figure}[ht!]
    \centering
    \begin{tikzpicture}[scale = 0.7]
    \fill [green, opacity  = 0.25] (0,1) rectangle (2,4);
    \draw (0,0)--(4,0)--(4,4)--(0,4)--(0,0);
    \draw (1,0) -- (1,4);
    \draw (2,0) -- (2,4);
    \draw (3,0) -- (3,4);
    \draw (0,1) -- (4,1);
    \draw (0,2) -- (4,2);
    \draw (0,3) -- (4,3);
    \node at (-0.5,0.5) {$4$};
    \node at (-0.5,1.5) {$3$};
    \node at (-0.5,2.5) {$2$};
    \node at (-0.5,3.5) {$1$};
    \node at (0.5,-0.5) {$1$};
    \node at (1.5,-0.5) {$2$};
    \node at (2.5,-0.5) {$3$};
    \node at (3.5,-0.5) {$4$};
    \node[orange] at (0.5,0.5) {$\star$};
    \node[orange] at (1.5,1.5) {$\star$};
    \node[orange] at (2.5,2.5) {$\star$};
    \node[orange] at (3.5,3.5) {$\star$};
    \node[blue] at (0.5,1.5) {$\bullet$};
    \node[blue] at (1.5,3.5) {$\bullet$};
    \node[blue] at (2.5,0.5) {$\bullet$};
    \node[blue] at (3.5,2.5) {$\bullet$};
    \draw[line width=0.55mm, red] (0,1) -- (2,1);
    \draw[line width=0.55mm, red] (1,2) -- (2,2);
    \draw[line width=0.55mm, red] (1,3) -- (3,3);
    \end{tikzpicture}
    \caption{$u = 4321,v = 3142$}
    \label{fig:4321-3142}
\end{figure}
\end{ex}

\begin{remark}
    If $u\leq v$ in strong Bruhat order, namely $u\leq_\mathbf{a} v$ where $\mathbf{a} = (1\ldots 1)$, then $\mathcal{T}_{u,v,\mathbf{a}}^\circ$ and $\mathcal{T}_{u,v,\mathbf{a}}$ are the (open) Richardson variety $\mathcal{R}_{u,v}^\circ$ and $\mathcal{R}_{u,v}$ respectively.
\end{remark}

\subsection{The second definition: pullback of cyclically rotated Richardson varieties}

For $V\in \gr(k,n)$, let $\widetilde{V}$ be any matrix such that $V$ is the column span of $\widetilde{V}$. Let $\chi:\gr(k,n) \rightarrow \gr(k,n)$ be the cyclic rotation such that for
    \begin{equation}\label{eqn:chi}
     \widetilde{V} = \begin{bmatrix}
    \text{---}  & \hspace{-0.2cm}\vec{v}_1 & \hspace{-0.2cm}\text{---} \\
    \text{---}  &\hspace{-0.2cm} \vec{v}_2 & \hspace{-0.2cm}\text{---} \\
    \ & \vdots & \ \\
    \text{---}  & \hspace{-0.2cm}\vec{v}_n & \hspace{-0.2cm}\text{---} \\
\end{bmatrix} \text{, set } 
    \chi(\widetilde{V}) := 
    \begin{bmatrix}
    \text{---}  & \hspace{-0.2cm}\vec{v}_n & \hspace{-0.2cm}\text{---} \\
    \text{---}  &\hspace{-0.2cm} \vec{v}_1 & \hspace{-0.2cm}\text{---} \\
    \ & \vdots & \ \\
    \text{---}  & \hspace{-0.2cm}\vec{v}_{n-1} & \hspace{-0.2cm}\text{---} \\
\end{bmatrix}.
\end{equation}
For $I = \{i_1,\dots,i_k\}\subset [n]$, denote $\chi(I):= \{i_1+1,\dots,i_k+1\}$, identifying $n+1$ with $1$. 

\begin{defin}
    For any $I,J\subset [n]$ with $|I|=|J|$ and $r\in[n]$ such that $I\leq_r J$, define the cyclically rotated (open) Grassmaniann Richardson variety
    \begin{align*}
        \mathcal{R}^{\circ}_{I,J,r}&:=\chi^{r-1}(\mathcal{R}^{\circ}_{\chi^{1-r}(I),\chi^{1-r}(J)}),\\
        \mathcal{R}_{I,J,r}&:=\chi^{r-1}(\mathcal{R}_{\chi^{1-r}(I),\chi^{1-r}(J)}).
    \end{align*}
\end{defin}

The cyclically rotated Richardson varieties are instances of Positroid varieties (see \cite[Section~6]{KLSjuggling}).
Similar to Grassmaniann Richardson varieties, cyclically rotated Richardson varieties can also be defined via vanishing of certain Pl\"ucker coordinates. 

\begin{prop}[\cite{KLSjuggling}]\label{prop:RIJadef}
For $I,J\subset [n]$, $r\in[n]$ with $|I|=|J|$ and $I\leq_r J$,
\[\mathcal{R}_{I,J,r} = \{V\in \gr(k,n): P_K(V) = 0 \text{ for }K\notin[I,J]_r\}.\]
The corresponding open cell $\mathcal{R}^\circ_{I,J,r} = \mathcal{R}_{I,J,r}\cap\{P_IP_J\neq 0\}$.
\end{prop}
\begin{proof}
    Following from \eqref{eqn:grSchubPlucker} and the fact that $\chi_\ast (P_K)=P_{\chi(K)}$,
    \[\mathcal{R}_{I,J,r} = \{V\in \gr(k,n):  P_{\chi^{r-1}(K)}(V) = 0 \text{ for }K\notin[\chi^{1-r}(I),\chi^{1-r}(J)]\}.\]
    The statement then follows from the fact that $A\leq_r B\iff \chi^{1-r}(A)\leq \chi^{1-r}(B)$.
\end{proof}
    
For each $k\in [n-1]$, define $\pi_k:\mathrm{Fl}_n \rightarrow \gr(k,n)$ to be the projection onto the $k$-th flag.

\begin{theorem}\label{thm:altdefT}
    For 
    $u\leq_{\mathbf{a}} v$, we have
    \[\mathcal{T}_{u,v,\mathbf{a}} = \bigcap_{k = 1}^{n-1}\pi_k^{-1}(\mathcal{R}_{u[k],v[k],a_k})\quad\text{ and }\quad\mathcal{T}^\circ_{u,v,\mathbf{a}} = \bigcap_{k = 1}^{n-1}\pi_k^{-1}(\mathcal{R}^\circ_{u[k],v[k],a_k}).\]
\end{theorem}
\begin{proof}
This is done by unpacking the rank conditions and comparing with \eqref{eqn:tiltRichOpen} and \eqref{eqn:tiltRich}.
\end{proof}

Since each $\pi_k^{-1}(\mathcal{R}_{u[k],v[k],a_k})$ is a closed subvariety of $\mathrm{Fl}_n$ and $\mathcal{R}^\circ_{u[k],v[k],a_k}\subseteq \mathcal{R}_{u[k],v[k],a_k}$ is open, we have the following corollary.

\begin{cor}
    $\mathcal{T}_{u,v,\mathbf{a}}$ is a closed subvariety of $\mathrm{Fl}_n$ and $\mathcal{T}_{u,v,\mathbf{a}}^\circ \subseteq \mathcal{T}_{u,v,\mathbf{a}}$ is open. 
\end{cor}

Let $\mathbf{a} = (a_1\dots a_{n-1})$ and $\mathbf{a}' = (a_1' \ldots a_{n-1}')$ be any two sequences such that $u\leq_{\mathbf{a}} v$ and $u\leq_{\mathbf{a}'}v$. Our next goal is to assert that the (open) tilted Richardson variety is independent of the choice of $\mathbf{a}$. We first need the following lemma. We note that this can be seen as an alternate interpretation of \Cref{lemma:both-r-comparable}.


\begin{lemma}\label{lemma:rotateRich}
   For $I,J \subset [n]$ with $|I| = |J|=k$, suppose $I\leq_r J$ and $I\leq_{r'}J$ for some $r\neq r'\in [n]$. Then $\mathcal{R}_{I,J,r} = \mathcal{R}_{I,J,r'}$ and $\mathcal{R}^\circ_{I,J,r} = \mathcal{R}^\circ_{I,J,r'}$. Furthermore, take any $V\in \mathcal{R}_{I,J,a}$ and its $n\times k$ matrix representative $M_V$ whose column span equals $V$. Then
    \[\mathrm{span}\{\vec{v}_i:i\in [r,r')_c\} \cap \mathrm{span}\{\vec{v}_i:i\in [r',r)_c\} = \{0\},\]
    where  $\vec{v_i}$ is the $i$-th row of $M_V$.
\end{lemma}
\begin{proof}
    Since $\mathcal{R}^\circ_{I,J,r}=\mathcal{R}_{I,J,r}\cap\{P_IP_J\neq 0\}$, we only need to prove the closed case. From Proposition~\ref{prop:RIJadef}, we have
    \begin{align*}
        \mathcal{R}_{I,J,r} &= \{V\in \gr(k,n):  P_K(V) = 0 \text{ for }K\notin[I,J]_r\},\\
        \mathcal{R}_{I,J,r'} &= \{V\in \gr(k,n):  P_K(V) = 0\text{ for }K\notin[I,J]_{r'}\}.
    \end{align*}
    Since Lemma~\ref{lemma:both-r-comparable} implies $[I,J]_r=[I,J]_{r'}$, we have $\mathcal{R}_{I,J,r}=\mathcal{R}_{I,J,r'}$.

    For the second part, Lemma~\ref{lemma:both-r-comparable}  implies that there exists some integer $0\leq d\leq k$ satisfying
    \begin{align*}
        \#(I\cap [r,r')_c)&=\#(J\cap [r,r')_c)=d,\\
        \#(I\cap [r',r)_c)&=\#(J\cap [r',r)_c)=k-d.
    \end{align*}
    By definition, for any $V\in \mathcal{R}^\circ_{I,J,r}$, we have
    \begin{align*}
        \#(I\cap [r,r')_c)=d &\implies \dim(\mathrm{span}\{\vec{v}_i:i\in [r,r')_c\})\leq d,\\
        \#(J\cap [r',r)_c)=k-d &\implies \dim(\mathrm{span}\{\vec{v}_i:i\in [r',r)_c\})\leq k-d.
    \end{align*}
    Notice that
    \[\dim(\mathrm{span}\{\vec{v}_i:i\in [r,r')_c\})+\dim(\mathrm{span}\{\vec{v}_i:i\in [r',r)_c\})\geq \dim(V)=k,\]
    and equality holds if and only if the two subspaces $\mathrm{span}\{\vec{v}_i:i\in [r,r')_c\}$ and $\mathrm{span}\{\vec{v}_i:i\in [r',r)_c\}$ are linearly independent. This concludes the proof.
\end{proof}

\begin{ex}
    Continuing Example~\ref{ex:4321-3142}, Lemma~\ref{lemma:rotateRich} implies that both the green shaded area and the purple hatched area in Figure~\ref{fig:rowspan-ind} have rank $1$. Moreover, the row spans of the two areas are independent.
    \begin{figure}[ht!]
    \centering
    \begin{tikzpicture}[scale = 0.7]
    \fill [green, opacity  = 0.25] (0,1) rectangle (2,3);
    \draw[pattern = north west lines, pattern color = purple, opacity = 0.5] (0,0) rectangle (2,1);
    \draw[pattern = north west lines, pattern color = purple, opacity = 0.5] (0,3) rectangle (2,4);
    \draw (0,0)--(4,0)--(4,4)--(0,4)--(0,0);
    \draw (1,0) -- (1,4);
    \draw (2,0) -- (2,4);
    \draw (3,0) -- (3,4);
    \draw (0,1) -- (4,1);
    \draw (0,2) -- (4,2);
    \draw (0,3) -- (4,3);
    \node at (-0.5,0.5) {$4$};
    \node at (-0.5,1.5) {$3$};
    \node at (-0.5,2.5) {$2$};
    \node at (-0.5,3.5) {$1$};
    \node at (0.5,-0.5) {$1$};
    \node at (1.5,-0.5) {$2$};
    \node at (2.5,-0.5) {$3$};
    \node at (3.5,-0.5) {$4$};
    \node[orange] at (0.5,0.5) {$\star$};
    \node[orange] at (1.5,1.5) {$\star$};
    \node[orange] at (2.5,2.5) {$\star$};
    \node[orange] at (3.5,3.5) {$\star$};
    \node[blue] at (0.5,1.5) {$\bullet$};
    \node[blue] at (1.5,3.5) {$\bullet$};
    \node[blue] at (2.5,0.5) {$\bullet$};
    \node[blue] at (3.5,2.5) {$\bullet$};
    \draw[line width=0.55mm, red] (1,1) -- (2,1);
    \draw[line width=0.55mm, red] (1,3) -- (2,3);
    \end{tikzpicture}
    \caption{The two horizontal red line segments represents two choices of $a_2$. Row spans of the green area and purple area are independent.}
    \label{fig:rowspan-ind}
\end{figure}
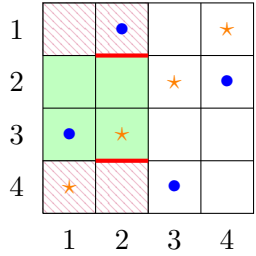
\end{ex}
\begin{cor}\label{cor:independent-a}
$\mathcal{T}_{u,v,\mathbf{a}}^\circ = \mathcal{T}_{u,v,\mathbf{a}'}^\circ$ and $\mathcal{T}_{u,v,\mathbf{a}} = \mathcal{T}_{u,v,\mathbf{a}'}$ for $u\leq_{\mathbf{a}} v$ and $u\leq_{\mathbf{a}'}v$.
\end{cor}
\begin{proof}
    This follows from Proposition~\ref{thm:altdefT} and Lemma~\ref{lemma:rotateRich}.
\end{proof}

As a consequence of \Cref{cor:independent-a}, we will denote the (open) tilted Richardson variety as $\mathcal{T}_{u,v}$ ($\mathcal{T}_{u,v}^\circ$). In particular, if $u\leq v$ in strong Bruhat order, $\mathcal{T}_{u,v}^\circ = \mathcal{R}_{u,v}^\circ$ and $\mathcal{T}_{u,v}=\mathcal{R}_{u,v}$.

\subsection{The third definition: vanishing of Pl\"ucker coordinates}

\begin{lemma}\label{lemma:plucker-set-to-perm}
    For any flag $F_\bullet\in \mathrm{Fl}_n$ and subset $I\in\binom{[n]}{k}$, if $P_I(F_\bullet)\neq 0$, then there exists a permutation $w\in S_n$ such that $w[k]=I$, and $P_w(F_\bullet)\neq0$.
\end{lemma}
\begin{proof}
    Let $M_F\in G$ be an matrix representative of $F_\bullet$ such that $F_i$ is the span of the first $i$ column vectors of $M_F$. The $k\times k$ submatrix $M'$ of $M_F$ in the rows indexed by $I$ in the first $k$ columns has full rank. Therefore, there exists a $(k-1)\times (k-1)$ submatrix of $M'$ in the first $(k-1)$ columns that also has full rank. Denote the row index of this submatrix as $I'$. Then $|I'|=k-1$ and $P_{I'}(F_\bullet)\neq 0$.

    Apply this process inductively to find a chain of subsets $\emptyset=I_0\subsetneq I_1\subsetneq\cdots\subsetneq I_{k-1}=I'\subsetneq I_k=I$ such that $P_{I_j}(F_\bullet)\neq 0$ for any $0\leq j\leq k$. Using similar idea, we can also find a chain $I=I_k\subsetneq I_{k+1}\subsetneq\cdots\subsetneq I_n=[n]$ such that $P_{I_j}(F_\bullet)\neq 0$ for $k\leq j\leq n$. Let $w\in S_n$ be the permutation such that $w_k=I_k\setminus I_{k-1}$. Then $P_w(F_\bullet)\neq 0$ and $w[k]=I$.
\end{proof}

\begin{theorem}\label{thm:altdefTplucker}
For any permutations $u,v\in S_n$
, we have
\[\mathcal{T}_{u,v} = \{F_\bullet\in \mathrm{Fl}_n: P_w(F_\bullet) = 0 \text{ for }w\notin[u,v]\},\]
    and the corresponding open cell $\mathcal{T}^\circ_{u,v} = \mathcal{T}_{u,v}\cap\{P_uP_v\neq 0\}$.
\end{theorem}

\begin{proof}
    Fix a sequence $\mathbf{a}$ such that $u\leq_{\mathbf{a}}v$. For any $F_\bullet\in \mathcal{T}_{u,v}$ and $w\notin[u,v]$, we first show $P_w(F_\bullet)=0$. By Theorem~\ref{thm:tilted-criterion}, there exists $k\in[n-1]$ such that $w[k]\notin[u[k],v[k]]_{a_k}$. Since $F_k\in\mathcal{R}_{u[k],v[k],a_k}$, by Proposition~\ref{prop:RIJadef}, $P_{w[k]}(F_k)=0$ and $P_w(F_\bullet)=0$. We have proven $\subseteq$.

    We now prove $\supseteq$. For any $F_\bullet\notin\mathcal{T}_{u,v,\mathbf{a}}$, there exists $k\in [n-1]$ such that $F_k\notin\mathcal{R}_{u[k],v[k],a_k}$. By Proposition~\ref{prop:RIJadef}, there exists subset $I\in{[n]\choose k}$ such that $I\notin [u[k],v[k]]_{a_k}$ and $P_I(F_k)\neq 0$. By Lemma~\ref{lemma:plucker-set-to-perm}, there exists a permutation $w\in S_n$ such that $P_w(F_\bullet)\neq0$ and $w[k]=I$. However, $I\notin [u[k],v[k]]_{a_k}$ implies $w\notin[u,v]$.

    For the open cell, notice that
    \[\mathcal{R}^\circ_{u[k],v[k],a_k}= \mathcal{R}_{u[k],v[k],a_k}\cap\{P_{u[k]}P_{v[k]}\neq 0\}.\]
    Therefore, by Theorem~\ref{thm:altdefT}, $\mathcal{T}^\circ_{u,v}=\mathcal{T}_{u,v}\cap\{P_uP_v=\prod_{k=1}^{n-1}P_{u[k]}P_{v[k]}\neq 0\}$ as desired.
\end{proof}
\section{Geometric properties of tilted Richardson varieties}\label{sec:tilted-Richardson-geometry}
In this section, we explore some further geometric properties of the tilted Richardson varieties, analogous to those of the classical Richardson varieties. We do note, however, that most of our analysis are significantly different and more technical than the classical case, due to the lack of an analogue notion of Schubert varieties and the constant usage of the sequence $\mathbf{a}=(a_1,\ldots,a_{n-1})$ that keeps varying.

\subsection{Stratification of $\mathcal{T}_{u,v}$ by open tilted Richardson varieties}
We start with a useful lemma that controls the choices of $a_k$ to some extent. 

\begin{lemma}\label{lemma:a_k=a_k+1}
    For any permutation $u,v$ and any $k\in [n-2]$, there exists $a_k$ such that $u[k]\leq_{a_k}v[k]$ and $u[k+1]\leq_{a_k} v[k+1]$.
\end{lemma}
\begin{proof}
    Consider the lattice path $P_k:= P(u[k],v[k])$.
    Since $u[k+1] = u[k]\cup\{u(k+1)\}$ and $v[k+1] = v[k]\cup \{v(k+1)\}$, we obtain $P_{k+1}:= P(u[k+1],v[k+1])$ from $P_k$ by,
    \begin{itemize}
        \item[-] at the $u(k+1)$-th step, changing a $\searrow$ to a $\rightarrow$ or changing a $\rightarrow$ to a $\nearrow$; and
        \item[-] at the $v(k+1)$-th step, changing a $\nearrow$ to a $\rightarrow$ or changing a  $\rightarrow$ to a $\searrow$.
    \end{itemize}
    In other words, we obtain $P_{k+1}$ by moving the lattice points 
    \begin{align}\label{eqn:Jul10aaa}
        (a,b)\in P_k \longrightarrow 
        \begin{cases}
            (a,b+1) &\text{ if }u(k+1)\leq a< v(k+1)\\
            (a,b-1) &\text{ if }v(k+1)\leq a< u(k+1)\\
            (a,b) & \text{ otherwise}
        \end{cases}.
    \end{align}
    Consider now $A_k:= \{a_k-1: u[k]\leq_{a_k} v[k]\}$. By Lemma~\ref{lemma:r-comparable}, $u[k]\leq_{a_k}v[k]$ if and only if $P_k$ hits the its lowest point with $x$-coordinate in $A_k$. Since the $y$-coordinate changes by at most $1$ from $P_k$ to $P_{k+1}$, if $A_{k+1}\cap A_{k} = \emptyset$, then all lattice points in $P_k$ with $x$-coordinate $a_k-1$ are moved up by $1$ and all lattice points with $x$-coordinate $a_{k+1}-1$ are moved down by $1$. This is impossible by \eqref{eqn:Jul10aaa}. Therefore $A_{k}\cap A_{k+1}\neq \emptyset$ and there exists $a_k$ such that $u[k]\leq_{a_k}v[k]$ and $u[k+1]\leq_{a_k} v[k+1]$.
\end{proof}
\begin{theorem}\label{thm:Tunion}
    $\mathcal{T}_{u,v} = \bigsqcup_{[x,y]\subseteq [u,v]}\mathcal{T}_{x,y}^\circ$.
\end{theorem}
\begin{proof}
    It is straightforward from definition that the right hand side is a disjoint union and that $\mathcal{T}_{x,y}^\circ\subset \mathcal{T}_{u,v}$ for all $[x,y]\subseteq [u,v]$. To prove the $\subsetneq$ direction, pick $F_\bullet \in \mathcal{T}_{u,v}$. We will show that $F_\bullet\in \mathcal{T}_{x,y}^\circ$ for some $[x,y]\subseteq [u,v]$.
    Pick $\mathbf{a}$ such that $u\leq_{\mathbf{a}} v$,
    for $k\in [n-1]$, set 
    \begin{align*}
    I_k =& \{i\in [n]: \dim(\mathrm{Proj}_{[a_k,j)_c}(F_k))< \dim(\mathrm{Proj}_{[a_k,j]_c}(F_k))\}, \\
    J_k =& \{j\in [n]: \dim(\mathrm{Proj}_{(i,a_k-1]_c}(F_k))< \dim(\mathrm{Proj}_{[i,a_k-1]_c}(F_k))\}.
    \end{align*}
    Then $F_k \in \mathcal{R}^\circ_{I_k,J_k,a_k}$.
    In particular, $I_k\leq_{a_k}J_k$ for all $k\in [n]$. Furthermore, $F_k\in \mathcal{R}_{u[k],v[k],a_k}$ implies that $[I,J]_{a_k}\subseteq [u[k],v[k]]_{a_k}$ is a subinterval in the shifted Gale order $\leq_{a_k}$.
    
    Fix $k\in [n-1]$, by Lemma~\ref{lemma:both-r-comparable} and \Cref{cor:independent-a}, the sets $I_k$ and $J_k$ are independent of the choice of $a_k$ as long as $u[k]\leq_{a_k} v[k]$. Combining with Lemma~\ref{lemma:a_k=a_k+1}, we can choose $\mathbf{a}$ such that $a_k = a_{k+1}$. Let $M\in {GL}_n$ be any matrix representative of $F_\bullet$ and let $M_k$ and $M_{k+1}$ denote the submatrix of $M$ consisting of the first $k$ and $k+1$ columns respectively. If
    \[\dim(\mathrm{Proj}_{[a_{k+1},j)_c}(F_{k+1})) = \dim(\mathrm{Proj}_{[a_{k+1},j]_c}(F_{k+1})),\]
    then row $j$ of $M_{k+1}$ lies in the row span of $M_{k+1}$ with row index in $[a_{k+1},j)_c$. Thus row $j$ of $M_{k}$ also lies in the row span of $M_{k}$ with row index in $[a_{k},j)_c$. This implies that $I_k\subset I_{k+1}$. The claim $J_k\subset J_{k+1}$ follows from similar reasoning. 
    
    Let $x,y\in S_n$ be permutations such that $x_k = I_k\setminus I_{k-1}$ and $y_k = J_k\setminus J_{k-1}$ for all $k\in [n]$. It then follows from definition of $I_k$ and $J_k$ that 
    $F_\bullet \in \mathcal{T}_{x,y}^\circ$ where $u\leq_{\mathbf{a}}x$ and $y\leq_{\mathbf{a}}v$. Since $x[k] = I_k\leq_{a_k} J_k = y[k]$ for all $k\in [n-1]$, namely $x\leq_\mathbf{a} y$, we conclude that any $F_\bullet\in \mathcal{T}_{u,v}$ lies in $\mathcal{T}_{x,y}^\circ$ for some $[x,y]\subset [u,v]$.
\end{proof}

\subsection{Dimension of $\mathcal{T}_{u,v}^\circ$ and $\mathcal{T}_{u,v}$}
In this subsection, we show that $\mathcal{T}_{u,v}^\circ$ and $\mathcal{T}_{u,v}$ have dimension $\ell(u,v)$ (\Cref{thm:uv}), the distance between $u,v$ in the quantum Bruhat graph (\Cref{sec:prelim}). We do this by analyzing $\mathcal{T}_{u,v}^{\circ}$ from different charts. 

For any permutation $x$, the permuted opposite Schubert cell $x\Omega_{id}^\circ$ is the set of flags
satisfying $P_x \neq 0$. These could also be realized as matrices with $1$'s in $(x_i,i)$ for all $i\in [n]$ and $0$'s to the right of $1$'s.

\begin{ex}\label{ex:chart}
    Let $x = 3142$, then $x\Omega_{id}^\circ$ can be identified as the following affine space with each $*$ being a copy of $\mathbb{C}$:
    \begin{center}
        \begin{tikzpicture}[scale = 0.7]
    \draw (0,0)--(4,0)--(4,4)--(0,4)--(0,0);
    \draw (1,0) -- (1,4);
    \draw (2,0) -- (2,4);
    \draw (3,0) -- (3,4);
    \draw (0,1) -- (4,1);
    \draw (0,2) -- (4,2);
    \draw (0,3) -- (4,3);
    \node at (-0.5,0.5) {$4$};
    \node at (-0.5,1.5) {$3$};
    \node at (-0.5,2.5) {$2$};
    \node at (-0.5,3.5) {$1$};
    \node at (0.5,-0.5) {$1$};
    \node at (1.5,-0.5) {$2$};
    \node at (2.5,-0.5) {$3$};
    \node at (3.5,-0.5) {$4$};
    \node at (0.5,1.5) {$1$};
    \node at (1.5,3.5) {$1$};
    \node at (2.5,0.5) {$1$};
    \node at (3.5,2.5) {$1$};
    \node at (1.5,1.5) {$0$};
    \node at (2.5,3.5) {$0$};
    \node at (2.5,1.5) {$0$};
    \node at (3.5,0.5) {$0$};
    \node at (3.5,3.5) {$0$};
    \node at (3.5,1.5) {$0$};
    \node at (2.5,2.5) {$*$};
    \node at (1.5,0.5) {$*$};
    \node at (1.5,2.5) {$*$};
    \node at (0.5,0.5) {$*$};
    \node at (0.5,2.5) {$*$};
    \node at (0.5,3.5) {$*$};
    \end{tikzpicture}
    \end{center}
\end{ex}
\begin{prop}\label{prop:uOmega}
    For any permutations $u,v\in S_n$,
    \begin{enumerate}
        \item if $x\notin [u,v]$, then $\mathcal{T}_{u,v} \cap x\Omega_{id}^\circ = \mathcal{T}_{u,v}^\circ \cap x\Omega_{id}^\circ = \emptyset$,
        \item $\mathcal{T}_{u,v}^\circ = \mathcal{T}_{u,v} \cap u\Omega_{id}^\circ \cap v\Omega_{id}^\circ$.
    \end{enumerate}
\end{prop}
\begin{proof}
    This follows directly from Theorem~\ref{thm:altdefTplucker}.
\end{proof}


\begin{defin}\label{def:tiltedrothe}
    For $u\leq_{\mathbf{a}} v$, define $D_{\mathbf{a}}^{\downarrow}(u), D_{\mathbf{a}}^{\uparrow}(v)\subset [n]^2$ to be the \emph{tilted Rothe diagrams}
    \begin{align*}
    D_{\mathbf{a}}^{\downarrow}(u) =& \{(i,k): i<_{a_k} u_k, u^{-1}(i)>k\},\\
    D_{\mathbf{a}}^{\uparrow}(v) =& \{(i,k): i>_{a_k}v_k, v^{-1}(i)>k\}.
    \end{align*}
\end{defin}
\begin{ex}\label{ex:tilteddiagram}
We demonstrate \Cref{def:tiltedrothe} with
$u = 4321, v= 3142$ and $\mathbf{a} = (4,4,2)$ using diagrams in \Cref{fig:Du} and \Cref{fig:Dv}. Put $\bullet$ in position $(u_k,k)$ and $(v_k,k)$. Draw a red horizontal line in column $k$ immediately above $a_k$. These red lines represent the ``floor" in each column. For $D_{\mathbf{a}}^{\downarrow}(u)$, draw death rays to the right and down of each $\bullet$ until they hit the red horizontal line or right boundary. Similarly, draw death rays to the right and up for $D_{\mathbf{a}}^{\uparrow}(v)$. The boxes remain are $D_{\mathbf{a}}^{\downarrow}(u)$ and $D_{\mathbf{a}}^{\uparrow}(v)$. 
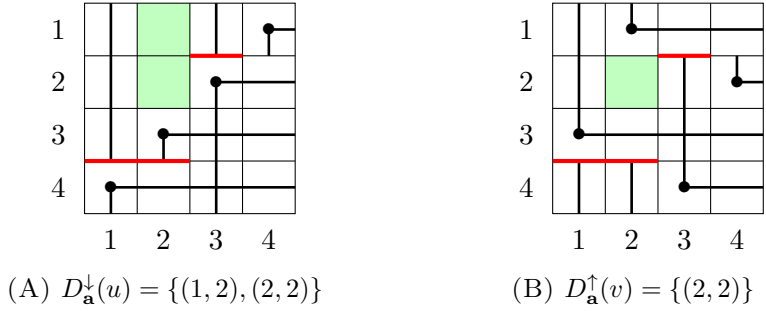
\begin{figure}[h!]
    \centering
    \subcaptionbox{$D_{\mathbf{a}}^{\downarrow}(u) = \{(1,2),(2,2)\}$ \label{fig:Du}}[.4\textwidth]{
        \begin{tikzpicture}[scale = 0.7]
    \fill [green, opacity  = 0.25] (1,4) rectangle (2,2);
    \draw (0,0)--(4,0)--(4,4)--(0,4)--(0,0);
    \draw (1,0) -- (1,4);
    \draw (2,0) -- (2,4);
    \draw (3,0) -- (3,4);
    \draw (0,1) -- (4,1);
    \draw (0,2) -- (4,2);
    \draw (0,3) -- (4,3);
    \node at (-0.5,0.5) {$4$};
    \node at (-0.5,1.5) {$3$};
    \node at (-0.5,2.5) {$2$};
    \node at (-0.5,3.5) {$1$};
    \node at (0.5,-0.5) {$1$};
    \node at (1.5,-0.5) {$2$};
    \node at (2.5,-0.5) {$3$};
    \node at (3.5,-0.5) {$4$};
    \node at (0.5,0.5) {$\bullet$};
    \node at (1.5,1.5) {$\bullet$};
    \node at (2.5,2.5) {$\bullet$};
    \node at (3.5,3.5) {$\bullet$};
    \draw[line width = 0.35mm] (0.5,0) -- (0.5,0.5) -- (4,0.5);
    \draw[line width = 0.35mm] (0.5,4) -- (0.5,1);
    \draw[line width = 0.35mm] (1.5,1) -- (1.5,1.5) -- (4,1.5);
    \draw[line width = 0.35mm] (2.5,0) -- (2.5,2.5) -- (4,2.5);
    \draw[line width = 0.35mm] (2.5,4) -- (2.5,3);
    \draw[line width = 0.35mm] (3.5,3) -- (3.5,3.5) -- (4,3.5);
    \draw[line width=0.55mm, red] (0,1) -- (2,1);
    \draw[line width=0.55mm, red] (2,3) -- (3,3);
    \end{tikzpicture}}
        \subcaptionbox{$D_{\mathbf{a}}^{\uparrow}(v) = \{(2,2)\}$ \label{fig:Dv}}[.4\textwidth]{
            \begin{tikzpicture}[scale = 0.7]
    \fill [green, opacity  = 0.25] (1,3) rectangle (2,2);
    \draw (0,0)--(4,0)--(4,4)--(0,4)--(0,0);
    \draw (1,0) -- (1,4);
    \draw (2,0) -- (2,4);
    \draw (3,0) -- (3,4);
    \draw (0,1) -- (4,1);
    \draw (0,2) -- (4,2);
    \draw (0,3) -- (4,3);
    \node at (-0.5,0.5) {$4$};
    \node at (-0.5,1.5) {$3$};
    \node at (-0.5,2.5) {$2$};
    \node at (-0.5,3.5) {$1$};
    \node at (0.5,-0.5) {$1$};
    \node at (1.5,-0.5) {$2$};
    \node at (2.5,-0.5) {$3$};
    \node at (3.5,-0.5) {$4$};
    \node at (0.5,1.5) {$\bullet$};
    \node at (1.5,3.5) {$\bullet$};
    \node at (2.5,0.5) {$\bullet$};
    \node at (3.5,2.5) {$\bullet$};
    \draw[line width = 0.35mm] (0.5,4) -- (0.5,1.5) -- (4,1.5);
    \draw[line width = 0.35mm] (0.5,0) -- (0.5,1);
    \draw[line width = 0.35mm] (1.5,4) -- (1.5,3.5) -- (4,3.5);
    \draw[line width = 0.35mm] (1.5,0) -- (1.5,1);
    \draw[line width = 0.35mm] (2.5,3) -- (2.5,0.5) -- (4,0.5);
    \draw[line width = 0.35mm] (3.5,3) -- (3.5,2.5) -- (4,2.5);
    \draw[line width=0.55mm, red] (0,1) -- (2,1);
    \draw[line width=0.55mm, red] (2,3) -- (3,3);
    \end{tikzpicture}}
        
    \end{figure}
\end{ex}

One may notice similarities between \Cref{fig:Du} and Rothe diagram of a permutation. Indeed, if $\mathbf{a} = (1,\dots,1)$, then $D_{\mathbf{a}}^{\downarrow}(u)$ is precisely the Rothe diagram $D(u^{-1})$. Moreover, on $u\Omega_{id}^\circ$, the vanishing of Pl\"ucker coordinates as in \Cref{def:plucker=0}(1) gives us the parametrization of $\Omega_u^\circ$ (see e.g. \cite[Section~2.2]{WooYong}).

\begin{defin}\label{def:plucker=0}
    For every pair of permutations $(u,v)$ and every choice of $\mathbf{a} = (a_1,\dots,a_{n-1})$ such that $u\leq_{\mathbf{a}}v$, define $\mathcal{I}_{u,v,\mathbf{a}}$ as the following set of equations:
\begin{enumerate}
    \item for every box $(i,k)\in D_{\mathbf{a}}^{\downarrow}(u)$, set $P_{u[k-1]+i}=0$,
    \item for every box $(i,k)\in D_{\mathbf{a}}^{\uparrow}(v)$, set $P_{v[k-1]+i}=0$.
\end{enumerate}

\begin{remark}
    An alternative way to phrase $\mathcal{I}_{u,v,\mathbf{a}}$ is, for column $k\in [n-1]$:
    \begin{itemize}
        \item under the chart $u\Omega_{id}^\circ$ as in Example~\ref{ex:chart}, set entries in the diagram $D^\downarrow_\mathrm{a}(u)$ to be $0$. 
        \item under the chart $v\Omega_{id}^\circ$, set entries in the diagram $D^\uparrow_\mathrm{a}(v)$ to be $0$. 
    \end{itemize}
\end{remark}
\end{defin}

Up until now, we have been working with any sequence $\mathbf{a}$ such that $u\leq_{\mathbf{a}}v$. There are many such choices for $\mathbf{a}$ and any one of them works. However, we are now going to distinguish some choices from the rest, that are particularly nice.
\begin{defin}
For a pair of permutations $u,v$, a sequence $\mathbf{a} = (a_1,\dots,a_{n-1})$ is \emph{flat} if 
    \begin{itemize}
        \item $u\leq_{\mathbf{a}} v$ and
        \item $u[k-1]\leq_{a_k} v[k-1]$ for all $k\in [2,n-1]$.
    \end{itemize}
\end{defin}
We note that this is equivalent to choosing $a_k\in (u_k,v_k]_c$ whenever possible. These are also the choices of $\mathbf{a}$'s such that the number of equations in $\mathcal{I}_{u,v,\mathbf{a}}$ is maximized. 
\begin{lemma}\label{lemma:flat}
For any pair of permutations $u,v$, there exists $\mathbf{a}$ that is flat. Moreover, for any such $\mathbf{a}$, it is also flat with respect to all pairs of $(x,y)$ such that $[x,y]\subseteq [u,v]$.
\end{lemma}
\begin{proof}
    This follows from Lemma~\ref{lemma:a_k=a_k+1} and Corollary~\ref{cor:xy}.
\end{proof}

\begin{lemma}\label{lemma:equationcount}
    For any $u,v$ and any flat $\mathbf{a}$, there are ${n\choose 2} - \ell(u,v)$ equations in $\mathcal{I}_{u,v,\mathbf{a}}$. 
\end{lemma}
\begin{proof}
We proceed by induction on $\ell(u,v)$. If $\ell(u,v) = 0$, we have $u = v$ and any sequence $\mathbf{a}$ is flat. For any $k\in [n-1]$ and any $a_k$, the number of entries in column $k$ of $D_{\mathbf{a}}^{\downarrow}(u)$ and $D_{\mathbf{a}}^{\uparrow}(v)$ sum up to $n-k$. Therefore $\mathcal{I}_{u,v,\mathbf{a}}$ consists of ${n\choose 2}$ equations.

    Suppose now that the statement holds when $\ell(u,v) = \ell$. For any pair $u,v$ such that $\ell(u,v) = \ell+1$, let $x = vt_{p,q}\in [u,v]$ be any permutation such that $\ell(u,x) = \ell$. Fix any flat $\mathbf{a}$ with respect to $u$ and $v$, by Lemma~\ref{lemma:flat}, $\mathbf{a}$ is also flat for $u,x$. 
    We are then left to show that there is one more equation from Definition~\ref{def:plucker=0} $(2)$ in $\mathcal{I}_{u,x,\mathbf{a}}$ than in $\mathcal{I}_{u,v,\mathbf{a}}$. We compare the diagrams $D_{\mathbf{a}}^{\uparrow}(v)$ and $D_{\mathbf{a}}^{\uparrow}(x)$ in each column $k$.

    For $k< p$ or $k> q$, since $v_k = x_k$ and $v[k] = x[k]$, 
    \begin{equation}\label{eqn:Aug19poi}
        (i,k)\in D_{\mathbf{a}}^{\uparrow}(v) \iff (i,k)\in D_{\mathbf{a}}^{\uparrow}(x).
    \end{equation}

    Consider now $k\in [p+1,q-1]$. In this case we still have $x_k = v_k$. 
        Since $x\rightarrow v = xt_{p,q}$ is an edge in quantum Bruhat graph, we have $x_k\notin [x_p,x_q]_c$. Since $x<_{\mathbf{a}}v$, for $k\in [p,q-1]$,
    $x[k]  <_{a_k} v[k]= x[k]\setminus \{x_p\}\cup \{x_q\}$. Therefore $x_p<_{a_k}x_q$ and thus $a_k\notin (x_p,x_q]_c$. 
    Therefore either both $x_p,x_q \in [x_k,a_{k}-1]_c$, in which case
    \begin{equation}\label{eqn:Aug19qwe}
        \{i: (i,k)\in D_{\mathbf{a}}^{\uparrow}(v)\} = \{i: (i,k)\in D_{\mathbf{a}}^{\uparrow}(x)\} \setminus \{x_p\} \cup \{x_q\},
    \end{equation}
    or $x_p,x_q\notin [x_k,a_{k}-1]_c$, in which case
    \begin{equation}\label{eqn:Aug20qwe}
        \{i: (i,k)\in D_{\mathbf{a}}^{\uparrow}(v)\} = \{i: (i,k)\in D_{\mathbf{a}}^{\uparrow}(x)\}.
    \end{equation}
    In both cases, $D_{\mathbf{a}}^{\uparrow}(v)$ and $D_{\mathbf{a}}^{\uparrow}(x)$ have the same number of boxes in column $k$.

    We are left to consider column $p$ and $q$. 
    Since $x_p<_{a_p} x_q$ and $x[p-1] = v[p-1]$,
    \begin{equation}\label{eqn:Aug19ggg}
        \{i: (i,p)\in D_{\mathbf{a}}^{\uparrow}(x)\} = \{i: (i,p)\in D_{\mathbf{a}}^{\uparrow}(v)\}\sqcup \{i\in [x_p,x_q]_c: x^{-1}(i)>p\}.
    \end{equation}
    Since $\mathbf{a}$ is flat, $x[q-1]<_{a_q}v[q-1] = x[q-1]\setminus \{x_p\} \cup \{x_q\}$ and thus $x_p<_{a_q} x_q$. Therefore 
    \begin{equation}\label{eqn:Aug19fff}
        \{i: (i,q)\in D_{\mathbf{a}}^{\uparrow}(v)\} = \{i: (i,q)\in D_{\mathbf{a}}^{\uparrow}(x)\}\sqcup \{i\in [x_p,x_q]_c: x^{-1}(i)>q \}.
    \end{equation}
    Since $x_k\notin [x_p,x_q]_c$ for all $k\in [p+1,q-1]$, 
    \[\{i\in [x_p,x_q]_c: x^{-1}(i)>p \}  = \{i\in [x_p,x_q]_c: x^{-1}(i)>q \} \cup \{x_q\}.\]
    Therefore $D_{\mathbf{a}}^{\uparrow}(x)\}$ has one more box than $D_{\mathbf{a}}^{\uparrow}(v)\}$ in column $p$ and $q$ combined.
    
    Combining Equations \eqref{eqn:Aug19poi}-\eqref{eqn:Aug19fff}, we conclude that there is one more equation from Definition~\ref{def:plucker=0} $(2)$ in $\mathcal{I}_{u,x,\mathbf{a}}$ than in $\mathcal{I}_{u,v,\mathbf{a}}$. We are then done by induction.
\end{proof}
\begin{ex}\label{ex:vx}
    We use the following example to help understand the proof of \Cref{lemma:equationcount}. Let $v = 465123, u = 263145, x = 265143$ and $\mathbf{a} = (2,2,2,6,6)$. One can verify that $x\in [u,v]$ and that $u\leq_{\mathbf{a}}v$. We draw $D_{\mathbf{a}}^{\uparrow}(v)$ and $D_{\mathbf{a}}^{\uparrow}(x)$ as in \Cref{ex:tilteddiagram}. Notice that both diagrams have the same number of boxes in  columns $2,3,4,6$, and the number of boxes in column $1,5$ combined differ by $1$.
    
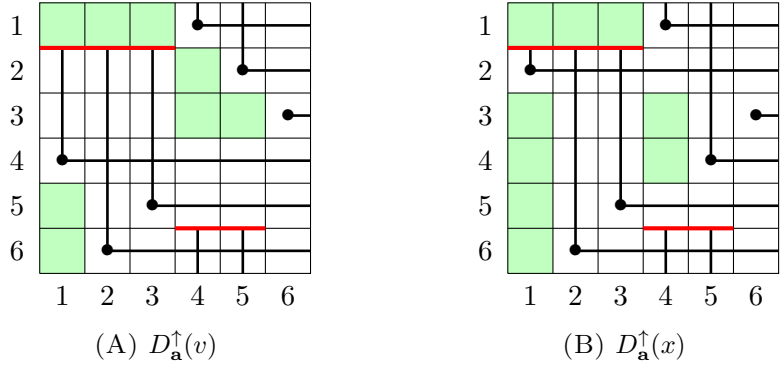
\begin{figure}[h!]
    \centering
    \subcaptionbox{$D_{\mathbf{a}}^{\uparrow}(v)$}[.4\textwidth]{
        \begin{tikzpicture}[scale = 0.6]
    \fill [green, opacity  = 0.25] (0,2) rectangle (1,0);
    \fill [green, opacity  = 0.25] (0,6) rectangle (3,5);
    \fill [green, opacity  = 0.25] (3,5) rectangle (4,3);
    \fill [green, opacity  = 0.25] (4,4) rectangle (5,3);
    \draw (0,0)--(6,0)--(6,6)--(0,6)--(0,0);
    \draw (1,0) -- (1,6);
    \draw (2,0) -- (2,6);
    \draw (3,0) -- (3,6);
    \draw (4,0) -- (4,6);
    \draw (5,0) -- (5,6);
    \draw (0,1) -- (6,1);
    \draw (0,2) -- (6,2);
    \draw (0,3) -- (6,3);
    \draw (0,4) -- (6,4);
    \draw (0,5) -- (6,5);
    \node at (-0.5,0.5) {$6$};
    \node at (-0.5,1.5) {$5$};
    \node at (-0.5,2.5) {$4$};
    \node at (-0.5,3.5) {$3$};
    \node at (-0.5,4.5) {$2$};
    \node at (-0.5,5.5) {$1$};
    \node at (0.5,-0.5) {$1$};
    \node at (1.5,-0.5) {$2$};
    \node at (2.5,-0.5) {$3$};
    \node at (3.5,-0.5) {$4$};
    \node at (4.5,-0.5) {$5$};
    \node at (5.5,-0.5) {$6$};
    \node at (0.5,2.5) {$\bullet$};
    \node at (1.5,0.5) {$\bullet$};
    \node at (2.5,1.5) {$\bullet$};
    \node at (3.5,5.5) {$\bullet$};
    \node at (4.5,4.5) {$\bullet$};
    \node at (5.5,3.5) {$\bullet$};
    \draw[line width = 0.35mm] (0.5,5) -- (0.5,2.5) -- (6,2.5);
    \draw[line width = 0.35mm] (1.5,5) -- (1.5,0.5) -- (6,0.5);
    \draw[line width = 0.35mm] (2.5,5) -- (2.5,1.5) -- (6,1.5);
    \draw[line width = 0.35mm] (3.5,6) -- (3.5,5.5) -- (6,5.5);
    \draw[line width = 0.35mm] (3.5,0) -- (3.5,1);
    \draw[line width = 0.35mm] (4.5,6) -- (4.5,4.5) -- (6,4.5);
    \draw[line width = 0.35mm] (5.5,3.5) -- (6,3.5);
    \draw[line width = 0.35mm] (4.5,0) -- (4.5,1);
    \draw[line width=0.55mm, red] (0,5) -- (3,5);
    \draw[line width=0.55mm, red] (3,1) -- (5,1);
    \end{tikzpicture}}
        \subcaptionbox{$D_{\mathbf{a}}^{\uparrow}(x)$ }[.4\textwidth]{
            \begin{tikzpicture}[scale = 0.6]
    \fill [green, opacity  = 0.25] (0,4) rectangle (1,0);
    \fill [green, opacity  = 0.25] (0,6) rectangle (3,5);
    \fill [green, opacity  = 0.25] (3,4) rectangle (4,2);
    \draw (0,0)--(6,0)--(6,6)--(0,6)--(0,0);
    \draw (1,0) -- (1,6);
    \draw (2,0) -- (2,6);
    \draw (3,0) -- (3,6);
    \draw (4,0) -- (4,6);
    \draw (5,0) -- (5,6);
    \draw (0,1) -- (6,1);
    \draw (0,2) -- (6,2);
    \draw (0,3) -- (6,3);
    \draw (0,4) -- (6,4);
    \draw (0,5) -- (6,5);
    \node at (-0.5,0.5) {$6$};
    \node at (-0.5,1.5) {$5$};
    \node at (-0.5,2.5) {$4$};
    \node at (-0.5,3.5) {$3$};
    \node at (-0.5,4.5) {$2$};
    \node at (-0.5,5.5) {$1$};
    \node at (0.5,-0.5) {$1$};
    \node at (1.5,-0.5) {$2$};
    \node at (2.5,-0.5) {$3$};
    \node at (3.5,-0.5) {$4$};
    \node at (4.5,-0.5) {$5$};
    \node at (5.5,-0.5) {$6$};
    \node at (0.5,4.5) {$\bullet$};
    \node at (1.5,0.5) {$\bullet$};
    \node at (2.5,1.5) {$\bullet$};
    \node at (3.5,5.5) {$\bullet$};
    \node at (4.5,2.5) {$\bullet$};
    \node at (5.5,3.5) {$\bullet$};
    \draw[line width = 0.35mm] (0.5,5) -- (0.5,4.5) -- (6,4.5);
    \draw[line width = 0.35mm] (1.5,5) -- (1.5,0.5) -- (6,0.5);
    \draw[line width = 0.35mm] (2.5,5) -- (2.5,1.5) -- (6,1.5);
    \draw[line width = 0.35mm] (3.5,6) -- (3.5,5.5) -- (6,5.5);
    \draw[line width = 0.35mm] (3.5,0) -- (3.5,1);
    \draw[line width = 0.35mm] (4.5,6) -- (4.5,2.5) -- (6,2.5);
    \draw[line width = 0.35mm] (5.5,3.5) -- (6,3.5);
    \draw[line width = 0.35mm] (4.5,0) -- (4.5,1);
    \draw[line width=0.55mm, red] (0,5) -- (3,5);
    \draw[line width=0.55mm, red] (3,1) -- (5,1);
    \end{tikzpicture}}
        
    \end{figure}
\end{ex}

\begin{prop}\label{prop:vanishing}
    For any $u,v$ and any flat $\mathbf{a}$,
    $\mathcal{T}^\circ_{u,v}=u\Omega_{id}^\circ \cap v\Omega_{id}^\circ\cap V(\mathcal{I}_{u,v,\mathbf{a}})$, where $V(\mathcal{I}_{u,v,\mathbf{a}})$ is the vanishing locus of equations in $\mathcal{I}_{u,v,\mathbf{a}}$.
\end{prop}
\begin{proof}
    We first show $\subseteq$. This is equivalent to proving that for every flag $F_\bullet\in \mathcal{T}^\circ_{u,v}$, $F_\bullet$ satisfies the equations in $\mathcal{I}_{u,v,\mathbf{a}}$. There are two type of equations in $\mathcal{I}_{u,v,\mathbf{a}}$:
    
    \noindent\textbf{Type (1).} For $(i,k)\in D^\downarrow_\mathbf{a}(u)$, since $i<_{a_k}u_k\implies u[k-1]\cup\{i\}<_{a_k}u[k]$, by Proposition~\ref{prop:RIJadef}
        \[F_\bullet\in\pi_k^{-1}(\mathcal{R}_{u[k],v[k],a_k})\implies P_{u[k-1]+i}(F_\bullet)=0;\]
        
    \noindent\textbf{Type (2).} For $(i,k)\in D^\uparrow_\mathbf{a}(v)$, since $i>_{a_k}v_k\implies v[k-1]\cup\{i\}>_{a_k}v[k]$, by Proposition~\ref{prop:RIJadef}
        \[F_\bullet\in\pi_k^{-1}(\mathcal{R}_{u[k],v[k],a_k})\implies P_{v[k-1]+i}(F_\bullet)=0.\]
    This implies $\mathcal{T}^\circ_{u,v}\subseteq V(\mathcal{I}_{u,v,\mathbf{a}})$. Combining with Proposition~\ref{prop:uOmega}, we are done.

    We now show $\supseteq$. It is enough to show that $\pi_k^{-1}(\mathcal{R}_{u[k],v[k],a_k})\supseteq V(\mathcal{I}_{u,v,\mathbf{a}})$ on the chart $u\Omega_{id}^\circ \cap v\Omega_{id}^\circ$ for all $k\in [n-1]$.
    We proceed by induction on $k$. Define $\mathcal{I}^k_{u,v,\mathbf{a}}\subseteq\mathcal{I}_{u,v,\mathbf{a}}$ to be the subset of equations that correspond to boxes in the $k$-th column of $D^\downarrow_\mathbf{a}(u)$ and $D^\uparrow_\mathbf{a}(v)$. For $k = 1$, $F_\bullet\in \pi_1^{-1}(\mathcal{R}_{u[k],v[k],a_k})$ if and only if $P_i(F_\bullet) = 0$ for all $i\notin [u_1,v_1]_c$. Notice that these are precisely the equations in $\mathcal{I}^1_{u,v,\mathbf{a}}$. Therefore $\pi_1^{-1}(\mathcal{R}_{u_1,v_1,a_1})\supseteq V(\mathcal{I}_{u,v,\mathbf{a}})$. 
    
    The induction step is to show that for any $F_\bullet \in u\Omega_{id}^\circ \cap v\Omega_{id}^\circ$:
    \[
         F_\bullet\in\pi_{k-1}^{-1}(\mathcal{R}_{u[k-1],v[k-1],a_{k-1}})\cap V(\mathcal{I}^k_{u,v,\mathbf{a}})\implies F_\bullet\in\pi_k^{-1}(\mathcal{R}_{u[k],v[k],a_k}).
    \]
    Since $\mathbf{a}$ is flat, by \Cref{lemma:rotateRich}, 
    \[\mathcal{R}_{u[k-1],v[k-1],a_{k-1}} = \mathcal{R}_{u[k-1],v[k-1],a_{k}}.\]
    We wish to show that if $I\not\geq_{a_k}u[k]$ or $I\not\leq_{a_k}v[k]$ then $P_I$ vanishes on $F_\bullet$. 
    We first fix any $I\not\geq_{a_k}u[k]$ and consider the incidence Pl\"ucker relation in \eqref{eqn:incidenceplucker1} associated to $u[k-1]$ and $I$:
    \begin{equation}\label{eqn:Aug21hhh}
        P_{u[k-1]}P_I=\sum_{i\in I}P_{I-i}P_{u[k-1]+i}.
    \end{equation}
    We claim that the right hand side of \eqref{eqn:Aug21hhh} vanishes. There are two cases:
    \begin{enumerate}
        \item If $i\geq_{a_k}u_k$. Since $I\not\geq_{a_k}u[k]$, we have $I\setminus \{i\}\not\geq_{a_k}u[k-1]$. This implies that $P_{I-i}$ vanishes since $F_\bullet\in\pi_{k-1}^{-1}(\mathcal{R}_{u[k-1],v[k-1],a_{k-1}})$;
        \item If $i<_{a_k}u_k$. Then $(i,k)$ is a box in $D^\downarrow_\mathbf{a}(u)$ and $P_{u[k-1]+i}$ vanishes.
    \end{enumerate}
    Therefore the both sides of    \eqref{eqn:Aug21hhh} vanish. Since $F_\bullet\in u\Omega^\circ_{id}$, $P_{u[k-1]}(F_\bullet)\neq 0$ and thus $P_I$ vanishes on $F_\bullet$. The case where $I\nleq_{a_k} v[k]$ follows from the same reasoning.
\end{proof}
    
\begin{theorem}\label{thm:uv}
    $\mathcal{T}^\circ_{u,v}$ is equidimensional of dimension $\ell(u,v)$.
\end{theorem}
\begin{proof}
    Since $\mathcal{T}^\circ_{u,v}$ is carved out by ${n\choose 2} - \ell(u,v)$ equations in Proposition~\ref{prop:vanishing}, by Krull's principal ideal theorem, every irreducible component of $\mathcal{T}^\circ_{u,v}$ has dimension at least $\ell(u,v)$.

    Now we need to prove $\dim(\mathcal{T}^\circ_{u,v})\leq\ell(u,v)$. If not, 
    let $Z\subseteq \mathcal{T}^\circ_{u,v}$ be an irreducible component with $\dim(Z)\geq \ell(u,v)+1$. Denote $\overline{Z}$ as the closure of $Z$ in $Fl_n$ and $\partial Z:=\overline{Z}\setminus Z$. Since $Z$ is closed in the chart $u\Omega_{id}^\circ \cap v\Omega_{id}^\circ$ (by Proposition~\ref{prop:vanishing}), we have $\partial Z=\overline{Z}\cap (Fl_n\setminus(u\Omega_{id}^\circ \cap v\Omega_{id}^\circ))=\overline{Z}\cap V(P_uP_v=0)$. There are two cases.
    \begin{enumerate}
        \item $\partial Z\neq \emptyset$. Since $\partial Z$ is the intersection of $\overline{Z}$ and a union of hyperplanes $V(P_uP_v=0)$, the dimensions of $\partial Z$ and $\overline{Z}$ differ by at most one: $\dim(\partial Z)\geq\dim(\overline{Z})-1\geq\ell(u,v)$. However, since $\partial Z\subseteq \partial\mathcal{T}_{u,v}(:=\mathcal{T}_{u,v}\setminus\mathcal{T}^\circ_{u,v})$ and $\partial\mathcal{T}_{u,v}=\sqcup_{[x,y]\subsetneq[u,v]}\mathcal{T}^\circ_{x,y}$by Theorem~\ref{thm:Tunion}, $\dim(\partial Z)\leq\ell(u,v)-1$ by induction on $\ell$, which is a contradiction.
        \item $\partial Z= \emptyset$. This implies that $Z=\overline{Z}$ is a projective variety. On the other hand, $Z\subseteq u\Omega_{id}^\circ \cap v\Omega_{id}^\circ$ is quasi-affine. 
        This cannot happen unless $Z=\{\text{pt}\}$, 
        contradicting $\dim(Z)\geq \ell(u,v)+1\geq 1$.
    \end{enumerate}
    Since we reach a contradiction in both cases, we conclude that $\dim(\mathcal{T}_{u,v}^\circ) = \ell(u,v)$.
\end{proof}



\subsection{Closure of $\mathcal{T}_{u,v}^\circ$}

Certainly, we expect $\overline{\mathcal{T}_{u,v}^\circ}=\mathcal{T}_{u,v}$ (\Cref{thm:closure}). One direction $\overline{\mathcal{T}_{u,v}^\circ}\subset\mathcal{T}_{u,v}$ is clear. For the other direction, by \Cref{thm:Tunion} and induction hypothesis, it suffices to show that $\mathcal{T}_{u,x}^{\circ}\subset\overline{\mathcal{T}_{u,v}^\circ}$ for $x\in[u,v]$ such that $\ell(u,x)=\ell(u,v)-1$, and dually, $\mathcal{T}_{y,v}^{\circ}\subset\overline{\mathcal{T}_{u,v}^\circ}$ for $y\in[u,v]$ such that $\ell(y,v)=\ell(u,v)-1$. We focus on such an $x$.

Throughout this section, let $x\in [u,v]$ such that $\ell(u,x) = \ell(u,v)-1$.
\begin{lemma}\label{lemma:uxOmega}
    $\mathcal{T}_{u,v}\cap u\Omega_{id}^\circ \cap x\Omega_{id}^\circ = \mathcal{T}_{u,x}^\circ \sqcup (\mathcal{T}_{u,v}^\circ \cap x\Omega_{id}^\circ)$.
\end{lemma}

\begin{proof}
    By Theorem~\ref{prop:uOmega}, the direction $\subseteq$ follows from the fact that there are only two strata in $\mathcal{T}_{u,v} = \bigsqcup_{[x,y]\subseteq [u,v]}\mathcal{T}_{x,y}^\circ$ that intersect both $u\Omega_{id}^\circ$ and $x\Omega_{id}^\circ$: $\mathcal{T}_{u,x}$ and $\mathcal{T}_{u,v}$. The other direction $\supseteq$ follows from $\mathcal{T}_{u,x}^\circ\subseteq u\Omega_{id}^\circ \cap x\Omega_{id}^\circ$ and $\mathcal{T}_{u,v}^\circ\subseteq u\Omega_{id}^\circ$ in Proposition~\ref{prop:uOmega}.
\end{proof}
The rest of this section is mainly devoted to controling the dimensions of this quasi-affine variety in \Cref{lemma:uxOmega} by finding explicit equations in said charts. 

\begin{defin}\label{def:uxplucker=0}
    For every pair of permutations $(u,v)$ and every choice of flat $\mathbf{a} = (a_1,\dots,a_{n-1})$ such that $u\leq_{\mathbf{a}}v$, let $x\in [u,v]$ such that $\ell(u,x) = \ell(u,v)-1$. Assume $x=vt_{pq}$ for $p<q$. Define $\mathcal{I}_{u,v,\mathbf{a},x}$ as the following set of equations:
\begin{enumerate}
    \item for every box $(i,k)\in D_{\mathbf{a}}^{\downarrow}(u)$, set $P_{u[k-1]+i}=0$,
    \item for every box $(i,k)\in D_{\mathbf{a}}^{\uparrow}(v)$ that is also in $D_{\mathbf{a}}^{\uparrow}(x)$, set $P_{x[k-1]+i}=0$,
    \item for every box $(i,k) \in D_{\mathbf{a}}^{\uparrow}(v)$ that is not in $D_{\mathbf{a}}^{\uparrow}(x)$ there are two cases:
    \begin{enumerate}
        \item $i=x_p$ and $k\in(p,q)$, set $P_{x[k-1]+x_q}=0$,
        \item $k=q$ and $i\in(x_p,x_q)_c$, set $P_{x[q-1]+i}P_{x[p-1]+x_q}-P_{x[p-1]+i}P_{x[q-1]+x_q} = 0$.
    \end{enumerate}
\end{enumerate}
\end{defin}

\begin{ex}
    We demonstrate equations of type $(2)$ and $(3)$ in \Cref{def:uxplucker=0} using $u,v,x,\mathbf{a}$ as in \Cref{ex:vx}. Here, $p = 1$ and $q = 5$. On $x\Omega_{id}^\circ$, equations from $(2)$ correspond to having $0$'s in the green shaded boxes in \Cref{fig:xOmega}. Equations from $(3a)$ and $(3b)$ correspond to having $0$ in position $(4,4)$ and vanishing of the $2\times 2$ minor in blue brackets in \Cref{fig:xOmega}.
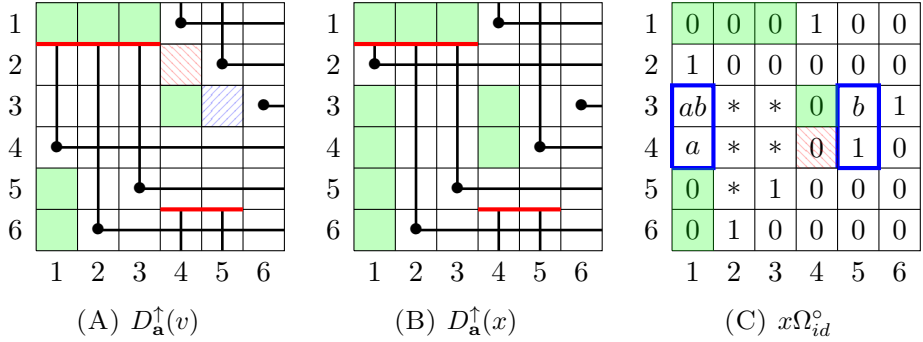
    \begin{figure}[h!]
    \centering
    \subcaptionbox{$D_{\mathbf{a}}^{\uparrow}(v)$ \label{fig:D_v}}[.4\textwidth]{
        \begin{tikzpicture}[scale = 0.55]
    \fill [green, opacity  = 0.25] (0,2) rectangle (1,0);
    \fill [green, opacity  = 0.25] (0,6) rectangle (3,5);
    \fill [green, opacity  = 0.25] (3,4) rectangle (4,3);
    \draw[pattern = north west lines, pattern color = red, opacity = 0.5] (3,5) rectangle (4,4);
    \draw[pattern = north east lines, pattern color = blue, opacity = 0.5] (4,3) rectangle (5,4);
    \draw (0,0)--(6,0)--(6,6)--(0,6)--(0,0);
    \draw (1,0) -- (1,6);
    \draw (2,0) -- (2,6);
    \draw (3,0) -- (3,6);
    \draw (4,0) -- (4,6);
    \draw (5,0) -- (5,6);
    \draw (0,1) -- (6,1);
    \draw (0,2) -- (6,2);
    \draw (0,3) -- (6,3);
    \draw (0,4) -- (6,4);
    \draw (0,5) -- (6,5);
    \node at (-0.5,0.5) {$6$};
    \node at (-0.5,1.5) {$5$};
    \node at (-0.5,2.5) {$4$};
    \node at (-0.5,3.5) {$3$};
    \node at (-0.5,4.5) {$2$};
    \node at (-0.5,5.5) {$1$};
    \node at (0.5,-0.5) {$1$};
    \node at (1.5,-0.5) {$2$};
    \node at (2.5,-0.5) {$3$};
    \node at (3.5,-0.5) {$4$};
    \node at (4.5,-0.5) {$5$};
    \node at (5.5,-0.5) {$6$};
    \node at (0.5,2.5) {$\bullet$};
    \node at (1.5,0.5) {$\bullet$};
    \node at (2.5,1.5) {$\bullet$};
    \node at (3.5,5.5) {$\bullet$};
    \node at (4.5,4.5) {$\bullet$};
    \node at (5.5,3.5) {$\bullet$};
    \draw[line width = 0.35mm] (0.5,5) -- (0.5,2.5) -- (6,2.5);
    \draw[line width = 0.35mm] (1.5,5) -- (1.5,0.5) -- (6,0.5);
    \draw[line width = 0.35mm] (2.5,5) -- (2.5,1.5) -- (6,1.5);
    \draw[line width = 0.35mm] (3.5,6) -- (3.5,5.5) -- (6,5.5);
    \draw[line width = 0.35mm] (3.5,0) -- (3.5,1);
    \draw[line width = 0.35mm] (4.5,6) -- (4.5,4.5) -- (6,4.5);
    \draw[line width = 0.35mm] (5.5,3.5) -- (6,3.5);
    \draw[line width = 0.35mm] (4.5,0) -- (4.5,1);
    \draw[line width=0.55mm, red] (0,5) -- (3,5);
    \draw[line width=0.55mm, red] (3,1) -- (5,1);
    \end{tikzpicture}}
        \hspace{-2cm}\subcaptionbox{$D_{\mathbf{a}}^{\uparrow}(x)$ \label{fig:D_x}}[.4\textwidth]{
            \begin{tikzpicture}[scale = 0.55]
    \fill [green, opacity  = 0.25] (0,4) rectangle (1,0);
    \fill [green, opacity  = 0.25] (0,6) rectangle (3,5);
    \fill [green, opacity  = 0.25] (3,4) rectangle (4,2);
    \draw (0,0)--(6,0)--(6,6)--(0,6)--(0,0);
    \draw (1,0) -- (1,6);
    \draw (2,0) -- (2,6);
    \draw (3,0) -- (3,6);
    \draw (4,0) -- (4,6);
    \draw (5,0) -- (5,6);
    \draw (0,1) -- (6,1);
    \draw (0,2) -- (6,2);
    \draw (0,3) -- (6,3);
    \draw (0,4) -- (6,4);
    \draw (0,5) -- (6,5);
    \node at (-0.5,0.5) {$6$};
    \node at (-0.5,1.5) {$5$};
    \node at (-0.5,2.5) {$4$};
    \node at (-0.5,3.5) {$3$};
    \node at (-0.5,4.5) {$2$};
    \node at (-0.5,5.5) {$1$};
    \node at (0.5,-0.5) {$1$};
    \node at (1.5,-0.5) {$2$};
    \node at (2.5,-0.5) {$3$};
    \node at (3.5,-0.5) {$4$};
    \node at (4.5,-0.5) {$5$};
    \node at (5.5,-0.5) {$6$};
    \node at (0.5,4.5) {$\bullet$};
    \node at (1.5,0.5) {$\bullet$};
    \node at (2.5,1.5) {$\bullet$};
    \node at (3.5,5.5) {$\bullet$};
    \node at (4.5,2.5) {$\bullet$};
    \node at (5.5,3.5) {$\bullet$};
    \draw[line width = 0.35mm] (0.5,5) -- (0.5,4.5) -- (6,4.5);
    \draw[line width = 0.35mm] (1.5,5) -- (1.5,0.5) -- (6,0.5);
    \draw[line width = 0.35mm] (2.5,5) -- (2.5,1.5) -- (6,1.5);
    \draw[line width = 0.35mm] (3.5,6) -- (3.5,5.5) -- (6,5.5);
    \draw[line width = 0.35mm] (3.5,0) -- (3.5,1);
    \draw[line width = 0.35mm] (4.5,6) -- (4.5,2.5) -- (6,2.5);
    \draw[line width = 0.35mm] (5.5,3.5) -- (6,3.5);
    \draw[line width = 0.35mm] (4.5,0) -- (4.5,1);
    \draw[line width=0.55mm, red] (0,5) -- (3,5);
    \draw[line width=0.55mm, red] (3,1) -- (5,1);
    \end{tikzpicture}}
    \hspace{-2cm}\subcaptionbox{$x\Omega_{id}^\circ$\label{fig:xOmega}}[.4\textwidth]{
        \begin{tikzpicture}[scale = 0.55]
            \draw (0,0)--(6,0)--(6,6)--(0,6)--(0,0);
    \draw (1,0) -- (1,6);
    \draw (2,0) -- (2,6);
    \draw (3,0) -- (3,6);
    \draw (4,0) -- (4,6);
    \draw (5,0) -- (5,6);
    \draw (0,1) -- (6,1);
    \draw (0,2) -- (6,2);
    \draw (0,3) -- (6,3);
    \draw (0,4) -- (6,4);
    \draw (0,5) -- (6,5);
    \node at (-0.5,0.5) {$6$};
    \node at (-0.5,1.5) {$5$};
    \node at (-0.5,2.5) {$4$};
    \node at (-0.5,3.5) {$3$};
    \node at (-0.5,4.5) {$2$};
    \node at (-0.5,5.5) {$1$};
    \node at (0.5,-0.5) {$1$};
    \node at (1.5,-0.5) {$2$};
    \node at (2.5,-0.5) {$3$};
    \node at (3.5,-0.5) {$4$};
    \node at (4.5,-0.5) {$5$};
    \node at (5.5,-0.5) {$6$};
    \node at (0.5,4.5) {$1$};
    \node at (1.5,0.5) {$1$};
    \node at (2.5,1.5) {$1$};
    \node at (3.5,5.5) {$1$};
    \node at (4.5,2.5) {$1$};
    \node at (5.5,3.5) {$1$};
    \node at (1.5,4.5) {$0$};
    \node at (2.5,4.5) {$0$};
    \node at (3.5,4.5) {$0$};
    \node at (4.5,4.5) {$0$};
    \node at (5.5,4.5) {$0$};
    \node at (2.5,0.5) {$0$};
    \node at (3.5,0.5) {$0$};
    \node at (4.5,0.5) {$0$};
    \node at (5.5,0.5) {$0$};
    \node at (3.5,1.5) {$0$};
    \node at (4.5,1.5) {$0$};
    \node at (4.5,5.5) {$0$};
    \node at (5.5,1.5) {$0$};
    \node at (5.5,2.5) {$0$};
    \node at (5.5,5.5) {$0$};

    \node at (0.5,0.5) {$0$};
    \node at (0.5,1.5) {$0$};
    \node at (0.5,2.5) {$a$};
    \node at (0.5,3.5) {$ab$};
    \node at (0.5,5.5) {$0$};
    \node at (1.5,1.5) {$*$};
    \node at (1.5,2.5) {$*$};
    \node at (1.5,3.5) {$*$};
    \node at (1.5,5.5) {$0$};
    \node at (2.5,2.5) {$*$};
    \node at (2.5,3.5) {$*$};
    \node at (2.5,5.5) {$0$};
    \node at (3.5,2.5) {$0$};
    \node at (3.5,3.5) {$0$};
    \node at (4.5,3.5) {$b$};
    \draw[line width=0.55mm, blue] (0,2) -- (1,2) -- (1,4) -- (0,4) -- (0,2);
    \draw[line width=0.55mm, blue] (4,2) -- (5,2) -- (5,4) -- (4,4) -- (4,2);
    \draw[pattern = north west lines, pattern color = red, opacity = 0.5] (3,3) rectangle (4,2);
    \fill [green, opacity  = 0.25] (0,2) rectangle (1,0);
    \fill [green, opacity  = 0.25] (0,6) rectangle (3,5);
    \fill [green, opacity  = 0.25] (3,4) rectangle (4,3);
        \end{tikzpicture}
        }
    \end{figure}
\end{ex}

\vspace{-0.35cm}

We need the following technical lemma and its consequence to prove Proposition~\ref{prop:uxvanishing}.

\begin{lemma}\label{lemma:indplucker}
    For any $I\in {[n]\choose a}, J\in {[n]\choose b}$ with $1\leq b < a < n$, if there exists $k\in (b,a)$ and $K_1,K_2\in {[n]\choose k}$ such that 
    $F_\bullet \in \mathcal{R}_{K_1,K_2,r} = \mathcal{R}_{K_1,K_2,r'}$ for some $r,r'\in [n]$, then
    \begin{equation}\label{eqn:indplucker}
        \sum_{i\in I \cap [r,r')_c} P_{I- i}P_{J+ i}(F_\bullet) = 0. \tag{$*$}
    \end{equation}
\end{lemma}
\begin{proof}
    Fix any $K\in {[n]\choose k}$ such that $P_K(F_\bullet) \neq 0$. Then Equation~\eqref{eqn:indplucker} is equivalent to 
    \[\sum_{i\in I \cap [r,r')_c} P_{I- i}P_{J+ i}P_K(F_\bullet) = 0.\]
    By expanding $P_{J+ i}P_K$ using \eqref{eqn:incidenceplucker3}, this is equivalent to
    \begin{equation}\label{eqn:Aug26aaa}
        \sum_{i\in I,i\in [r,r')_c, j\in K}P_{I-i}P_{J+j}P_{K-j+i}(F_\bullet) = 0.
    \end{equation}
    Since $K_1\leq_{r} K_2$ and $K_1\leq_{r'} K_2$, $P_{K'}(F_\bullet)\neq 0$ only if $\#(K'\cap [r,r')_c) = \#(K_1\cap [r,r')_c)$ for any $K'\in {[n]\choose k}$. Thus $P_{K-j+i}(F_\bullet)\neq 0$ only if $j\in [r,r')_c$. Equation~\eqref{eqn:Aug26aaa} is then equivalent to 
    \[\sum_{i\in I\cap [r,r')_c, j\in K\cap[r,r')_c}P_{I-i}P_{J+j}P_{K-j+i}(F_\bullet) = 0.\]
    By the same reasoning, this is equivalent to 
    \begin{equation}\label{eqn:Aug26bbb}
        \sum_{i\in I, j\in K\cap[r,r')_c}P_{I-i}P_{J+j}P_{K-j+i}(F_\bullet) = 0.
    \end{equation}
    Since $\sum_{i\in I}P_{I-i}P_{K-j+i} = 0$ for $j\in K\cap[r,r')_c$ by \eqref{eqn:incidenceplucker2}, \Cref{eqn:Aug26bbb} and thus \eqref{eqn:indplucker} holds. 
\end{proof}

\begin{cor}\label{cor:induxplucker}
    Given 
    $u\leq_\mathbf{a} v$, for any $I\in {[n]\choose q}, J\in {[n]\choose p-1}$ with $1\leq p < q < n$, if flag $F_\bullet$ satisfies $F_\bullet \in\pi_{k}^{-1}(\mathcal{R}_{u[k],v[k],a_k})$ for every $p\leq k<q$, then
    \begin{equation}\label{eqn:induxplucker}
        \sum_{\substack{i\in I\cap [a_p,a_q)_c}}P_{I- i}P_{J+ i}(F_\bullet) = 0. \tag{$*$}
    \end{equation}
\end{cor}
\begin{proof}
    By \Cref{lemma:flat}, there is a flat $\mathbf{a'}=(a_1',\dots,a_{n-1}')$ such that $u\leq_{\mathbf{a'}}v$. By \Cref{lemma:indplucker}, 
    \[\sum_{\substack{i\in I\cap [a_p,a_p')_c}}P_{I- i}P_{J+ i}(F_\bullet) = 0\text{ and }\sum_{\substack{i\in I\cap [a_q',a_q)_c}}P_{I- i}P_{J+ i}(F_\bullet) = 0.\]
    Moreover, for every $p\leq k<q$, by flatness of $\mathbf{a}'$:
    \begin{equation*}
        \sum_{\substack{i\in I\cap [a_k',a_{k+1}')_c}}P_{I- i}P_{J+ i}(F_\bullet) = 0;
    \end{equation*}
    By summing the above equations and subtracting $\sum_{i\in I}P_{I-i}P_{J+i}(F_\bullet) = 0$ appropriate times, we obtain \eqref{eqn:induxplucker}.
\end{proof}

\begin{prop}\label{prop:uxvanishing}
    For any flat $\mathbf{a}$, $\mathcal{T}_{u,v}\cap u\Omega_{id}^\circ \cap x\Omega_{id}^\circ=u\Omega_{id}^\circ \cap x\Omega_{id}^\circ\cap V(\mathcal{I}_{u,v,\mathbf{a},x})$. 
\end{prop}
\begin{proof}
    We first show $\subseteq$. This is equivalent to proving that $F_\bullet$ satisfies the equations in $\mathcal{I}_{u,v,\mathbf{a},x}$ if $F_\bullet\in \mathcal{T}_{u,v}\cap u\Omega_{id}^\circ \cap x\Omega_{id}^\circ$. There are four types of equations in $\mathcal{I}_{u,v,\mathbf{a},x}$:

\noindent\textbf{Type (1).} Here $(i,k)\in D^\downarrow_\mathbf{a}(u)$. Since $i<_{a_k}u_k$, we have $u[k-1]\cup\{i\}<_{a_k}u[k]$, by Proposition~\ref{prop:RIJadef}
        \[F_\bullet\in\pi_k^{-1}(\mathcal{R}_{u[k],v[k],a_k})\implies P_{u[k-1]+i}(F_\bullet)=0;\]

\noindent\textbf{Type (2).} Here $(i,k)\in D^\uparrow_\mathbf{a}(v)\cap D^\uparrow_\mathbf{a}(x)$. Since $i>_{a_k}v_k$ and $x_p>_{a_k}v_p$ if $k\in(p,q]$, we have $x[k-1]\cup\{i\}>_{a_k}v[k]$, and
        \[F_\bullet\in\pi_k^{-1}(\mathcal{R}_{u[k],v[k],a_k})\implies P_{x[k-1]+i}(F_\bullet)=0.\]

\noindent\textbf{Type (3a).} Here $(x_p,k)\in D^\uparrow_\mathbf{a}(v)\setminus D^\uparrow_\mathbf{a}(x)$ where $i = x_p$ and $k\in(p,q)$. Since $x_p>_{a_k}v_k$, we have $x[k-1]\cup\{x_q\}>_{a_k}v[k]$, and
        \[F_\bullet\in\pi_k^{-1}(\mathcal{R}_{u[k],v[k],a_k})\implies P_{x[k-1]+x_q}(F_\bullet)=0.\]

\noindent\textbf{Type (3b).} Here $(i,q)\in D^\uparrow_\mathbf{a}(v)\setminus D^\uparrow_\mathbf{a}(x)$ where $k = q$ and $i\in(x_p,x_q)_c$. By \eqref{eqn:incidenceplucker3}, we have
        \begin{equation}\label{eqn:Aug27aaa}
            P_{x[p-1]+i}P_{v[q]}=-\sum_{k\in[p,q]}P_{x[p-1]+v_k}P_{v[q]-v_k+i}.
        \end{equation}
The left hand side of \eqref{eqn:Aug27aaa} is $P_{x[p-1]+i}P_{v[q]} = (-1)^{\#x[q-1]>x_q}P_{x[p-1]+i}P_{x[q-1]+x_q}$.
        
        For $k=p$, the corresponding term in right hand side of \eqref{eqn:Aug27aaa} is $$-P_{x[p-1]+v_p}P_{v[q]-v_p+i} = (-1)^{\#x[q-1]>x_q}P_{x[p-1]+x_q}P_{x[q-1]+i}.$$ 
        
        For $k = q$, since $i\in (x_p,x_q)_c$ and $a_q\notin (x_p,x_q)_c$, we have $i>_{a_q}v_q$ and thus $v[q]-v_q+i>_{a_q}v[q]$. Therefore $F_\bullet\in \pi_q^{-1}(\mathcal{R}_{u[q],v[q],a_q})\implies P_{v[q]-v_q+i}(F_\bullet) = 0$ for all $i\in (x_p,x_q)_c$. 

        For $p<k<q$, if $v_k>_{a_p}v_p$, then $x[p-1]+v_k >_{a_p} v[p]$ and $P_{x[p-1]+v_k}(F_\bullet) = 0$. Similarly, if $i>_{a_q} v_k$, then $v[q]-v_k+i >_{a_q} v[q]$ and $P_{v[q]-v_k+i}(F_\bullet) = 0$. Therefore, for each $i$, 
        \begin{equation}\label{eqn:Aug27zzz}
            \sum_{k\in(p,q)}P_{x[p-1]+v_k}P_{v[q]-v_k+i} = \sum_{k\in (p,q), i\leq_{a_q} v_k \leq_{a_p} v_p} P_{x[p-1]+v_k}P_{v[q]-v_k+i}.
        \end{equation}
        Since $v_k \notin [x_p,x_q]_c$ for $k\in (p,q)$, 
        \[i\leq_{a_q} v_k \iff v_k \in (x_q,a_q)_c \text{ and }v_k \leq_{a_p} v_p \iff v_k\in [a_p,x_q)_c.\] 
        The right hand side of \eqref{eqn:Aug27zzz} equals
        \[\sum_{k:v_k\in [a_p,a_q)_c} P_{x[p-1]+v_k}P_{v[q]-v_k+i}.\]
        Since $i\notin [a_p,a_q)_c$, this is $0$ by \Cref{cor:induxplucker}. 
        Therefore
        $$F_\bullet \in \mathcal{T}_{u,v}\cap u\Omega_{id}^\circ\cap x\Omega_{id}^\circ\implies P_{x[q-1]+i}P_{x[p-1]+x_q}-P_{x[p-1]+i}P_{x[q-1]+x_q}(F_\bullet) = 0.$$
Since we have exhausted all equations in $\mathcal{I}_{u,v,\mathbf{a},x}$, we conclude the proof of $\subseteq$ direction.

We now prove the $\supseteq$ direction. Namely for $F_\bullet \in u\Omega_{id}^\circ \cap x\Omega_{id}^\circ$ and $k\in [n-1]$, if $F_\bullet \in V(\mathcal{I}_{u,v,\mathbf{a},x})$ then $F_\bullet \in \pi_k^{-1}(\mathcal{R}_{u[k],v[k],a_k})$. By the same reasoning in \Cref{prop:vanishing}, if $F_\bullet \in V(\mathcal{I}_{u,v,\mathbf{a},x})\cap u\Omega_{id}^\circ$ then $P_{K}(F_\bullet) = 0$ for all $K\in {[n]\choose k}$ such that $K\ngeq_{a_k} u[k]$. We can then focus on Equations in $\mathcal{I}_{u,v,\mathbf{a},x}$ of type $(2)$ and $(3)$ as in \Cref{def:uxplucker=0}. It is enough to show that if $F_\bullet \in x\Omega_{id}^\circ \cap V(\mathcal{I}_{u,v,\mathbf{a},x})$ then $P_K(F_\bullet) = 0$ for all $K\nleq_{a_k} v[k]$. 

We proceed by induction on $k$. For $k = 1$, $P_{i}\in \mathcal{I}_{u,v,\mathbf{a},x}\iff i>_{a_1}v_1$. Therefore $F_\bullet \in V(\mathcal{I}_{u,v,\mathbf{a},x})\implies F_\bullet \in \pi_1^{-1}(\mathcal{R}_{u_1,v_1,a_1})$. Suppose the statement holds for $k-1$, we fix $K\in {[n]\choose k}$ such that $K\nleq v[k]$ and 
consider the following two cases:

\

\noindent \textbf{Case \RNum{1}} ($k\neq q$): By \eqref{eqn:incidenceplucker1}, we have
\begin{equation}\label{eqn:Sep7aaa}
    P_{x[k-1]}P_{K} = \sum_{i\in K}P_{x[k-1]+i}P_{K-i}.
\end{equation}
If $i>_{a_k}v_k$, then $P_{x[k-1]+i}(F_\bullet) = 0$ by $(2)$ of \Cref{def:uxplucker=0}. If $i\leq_{a_k}v_k$, then $K-i\nleq_{a_k} v[k-1]$. Since $\mathbf{a}$ is flat, $K-i\nleq_{a_{k-1}} v[k-1]$ and thus $P_{K-i}(F_\bullet) = 0$ by inductive hypothesis. The right hand side of \eqref{eqn:Sep7aaa} is then $0$. Since $F_\bullet \in x\Omega_{id}^\circ$, $P_{x[k-1]}(F_\bullet) \neq 0$. Thus $P_K(F_\bullet) = 0$.



\

\noindent \textbf{Case \RNum{2}} ($k=q$):
Again by \eqref{eqn:incidenceplucker1}, 
\begin{equation}\label{eqn:Aug28aaa}
    P_{x[q-1]}P_{K} = \sum_{i\in K} P_{x[q-1]+i}P_{K-i}.
\end{equation}
If $i>_{a_q} x_q$, then either $i\in x[q-1]$ or $(i,q)\in D_{\mathbf{a}}^{\uparrow}(v)\cap D_{\mathbf{a}}^{\uparrow}(x)$. In both cases, we have $P_{x[q-1]+i}(F_\bullet) = 0$. If $i\leq_{a_q} x_p = v_q$, then $K-i \nleq v[q-1]$ and $P_{K-i}(F_\bullet) = 0$ by inductive hypothesis. The right hand side of \eqref{eqn:Aug28aaa} then equals
\begin{equation}\label{eqn:Aug28zzz}
    \sum_{i\in K\cap (x_p,x_q]_c}P_{x[q-1]+i}P_{K-i}.
\end{equation}
In order to show the above equation vanishes, we wish to multiply by $P_{x[p-1]+x_q}$ and apply the equation in $(3b)$ of \Cref{def:uxplucker=0}. However, it is not guaranteed that $P_{x[p-1]+x_q}(F_\bullet)\neq 0$ for $F_\bullet \in x\Omega_{id}^\circ\cap u\Omega_{id}^\circ \cap V(\mathcal{I}_{u,v,\mathbf{a},x})$. So we further divide into two cases.

\noindent \textbf{Case \RNum{2a}} ($P_{x[p-1]+x_q}(F_\bullet)\neq 0$): Using $(3b)$ from \Cref{def:uxplucker=0}, we have
\begin{equation}\label{eqn:Aug28bbb}
    \sum_{i\in K} P_{x[p-1]+x_q}P_{x[q-1]+i}P_{K-i} = \sum_{i\in K} P_{x[q-1]+x_q}P_{x[p-1]+i}P_{K-i}.
\end{equation}
By \eqref{eqn:incidenceplucker2}, \Cref{eqn:Aug28bbb} equals $0$. Since $P_{K-i}(F_\bullet) = 0$ for $i\in [a_q,x_p]_c$ and $P_{x[p-1]+i}(F_\bullet) = 0$ for $i\in (x_q,a_p)_c$, we have
\[\sum_{i\in K\cap([a_p,a_q)_c\cup (x_p,x_q]_c)} P_{x[q-1]+x_q}P_{x[p-1]+i}P_{K-i} = 0.\]
Since $[a_p,a_q)_c$ and $(x_p,x_q]_c$ are disjoint, by \Cref{cor:induxplucker}, we have 
\[\sum_{i\in K\cap (x_p,x_q]_c}P_{x[q-1]+x_q}P_{x[q-1]+i}P_{K-i} = 0.\]
Since $P_{x[p-1]+x_q}(F_\bullet),P_{x[q-1]}(F_\bullet)\neq 0$, by \eqref{eqn:Aug28aaa} and \eqref{eqn:Aug28zzz}, we must have $P_K(F_\bullet) = 0$. 

\noindent \textbf{Case \RNum{2b}} ($P_{x[p-1]+x_q}(F_\bullet)= 0$): For $i\in(x_p,x_q)_c$, since we have $P_{x[q-1]+i}P_{x[p-1]+x_q}=P_{x[p-1]+i}P_{x[q-1]+x_q}$, we have $P_{x[p-1]+i}(F_\bullet) = 0$. Notice by \Cref{def:plucker=0} and \Cref{def:uxplucker=0}, these are the equations in $\mathcal{I}_{u,x,\mathbf{a}}\setminus \mathcal{I}_{u,v,\mathbf{a},x}$. Therefore
$$F_\bullet\in V(\mathcal{I}_{u,x,\mathbf{a}})\cap x\Omega_{id}^\circ\cap u\Omega_{id}^\circ = \mathcal{T}_{u,x}^\circ.$$
By \Cref{lemma:uxOmega}, we have $V(\mathcal{I}_{u,v,\mathbf{a},x})\cap x\Omega_{id}^\circ\cap u\Omega_{id}^\circ\subseteq \mathcal{T}_{u,v}\cap x\Omega_{id}^\circ\cap u\Omega_{id}^\circ$ as desired.
\end{proof}

\begin{prop}\label{prop:ux}
    $\mathcal{T}_{u,v}\cap u\Omega_{id}^\circ \cap x\Omega_{id}^\circ$ is equidimensional of dimension $\ell(u,v)$. 
\end{prop}
\begin{proof}
    Since $\mathcal{T}_{u,v}\cap u\Omega_{id}^\circ \cap x\Omega_{id}^\circ$ is the zero locus of ${n\choose 2}-\ell(u,v)$ equations by Proposition~\ref{prop:uxvanishing}, every irreducible component has dimension $\geq\ell(u,v)$. On the other hand, \Cref{thm:uv} and \Cref{lemma:uxOmega} implies that every irreducible component has dimension at most $\ell(u,v)$.
\end{proof}


\begin{theorem}\label{thm:closure}
    $\mathcal{T}_{u,v} = \overline{\mathcal{T}_{u,v}^\circ}$.
\end{theorem}
\begin{proof}
    We prove by induction on $\ell(u,v)$. For $\ell(u,v)=0$, $u=v$ and $\mathcal{T}_{u,u} = \mathcal{T}_{u,u}^\circ = \{e_u\}$. Now assume $\ell(u,v)>0$. By Theorem~\ref{thm:Tunion}, it is enough to show that $\mathcal{T}_{u,x}^\circ \subseteq \overline{\mathcal{T}_{u,v}^\circ}$ for all $x\in [u,v]$ such that $\ell(u,x) = \ell(u,v)-1$. Pick any such $x$, by \Cref{lemma:uxOmega}, we have
    \[\overline{\mathcal{T}_{u,v}\cap u\Omega_{id}^\circ\cap x\Omega_{id}^\circ} = \overline{\mathcal{T}_{u,x}^\circ} \cup \overline{\mathcal{T}^\circ_{u,v}\cap x\Omega_{id}^\circ}\]
    as quasi-affine varieties in $u\Omega_{id}^\circ\cap x\Omega_{id}^\circ$.
    Let $Z\subseteq \overline{\mathcal{T}_{u,v}\cap u\Omega_{id}^\circ \cap x\Omega_{id}^\circ}$ be any irreducible component. Since $\dim(Z) = \ell(u,v)>\dim(\overline{\mathcal{T}_{u,x}^\circ})$ by Proposition~\ref{prop:ux}, $Z \nsubseteq \overline{\mathcal{T}_{u,x}^\circ}$. Since $Z$ is irreducible, $Z\subseteq \overline{\mathcal{T}^\circ_{u,v}\cap x\Omega_{id}^\circ}$, and since this holds for all irreducible components $Z$,
    $\overline{\mathcal{T}_{u,v}\cap u\Omega_{id}^\circ\cap x\Omega_{id}^\circ} \subseteq \overline{\mathcal{T}^\circ_{u,v}\cap x\Omega_{id}^\circ}$. Thus, $\mathcal{T}^\circ_{u,x}\subseteq \overline{\mathcal{T}^\circ_{u,x}}\subseteq\overline{\mathcal{T}^\circ_{u,v}\cap x\Omega_{id}^\circ}\subseteq\overline{\mathcal{T}_{u,v}^\circ}$ as desired. 
\end{proof}


\section*{Acknowledgements}
We are grateful to Alex Postnikov for introducing us to quantum Bruhat graphs. We thank Anders Buch, Allen Knutson, Thomas Lam, Leonardo Mihalcea, Lauren Williams, Weihong Xu and Alexander Yong for many helpful conversations. SG is partially supported by an NSF Graduate Research Fellowship under grant No. DGE-1746047.

\bibliographystyle{plain}
\bibliography{ref}
\end{document}